\documentclass[11pt]{amsart}
\usepackage{amsopn}
\usepackage{amssymb, amscd}
\usepackage{graphicx, graphics, epsfig}

\newcommand{\nc}{\newcommand}

\nc{\eg}{\mathfrak{e} } \nc{\fg}{\mathfrak{f} } \nc{\vg}{\mathfrak{v} } \nc{\wg}{\mathfrak{w} }
\nc{\zg}{\mathfrak{z} } \nc{\ngo}{\mathfrak{n} } \nc{\kg}{\mathfrak{k} }
\nc{\mg}{\mathfrak{m} } \nc{\bg}{\mathfrak{b} } \nc{\ggo}{\mathfrak{g} }
\nc{\ggob}{\overline{\mathfrak{g}} } \nc{\sog}{\mathfrak{so} }
\nc{\sug}{\mathfrak{su} } \nc{\spg}{\mathfrak{sp} } \nc{\slg}{\mathfrak{sl} }
\nc{\glg}{\mathfrak{gl} } \nc{\cg}{\mathfrak{c} } \nc{\rg}{\mathfrak{r} }
\nc{\hg}{\mathfrak{h} } \nc{\tg}{\mathfrak{t} } \nc{\ug}{\mathfrak{u} }
\nc{\dg}{\mathfrak{d} } \nc{\ag}{\mathfrak{a} } \nc{\pg}{\mathfrak{p} }
\nc{\sg}{\mathfrak{s} } \nc{\affg}{\mathfrak{aff} }

\nc{\pca}{\mathcal{P}} \nc{\nca}{\mathcal{N}} \nc{\lca}{\mathcal{L}}
\nc{\oca}{\mathcal{O}} \nc{\mca}{\mathcal{M}} \nc{\tca}{\mathcal{T}}
\nc{\aca}{\mathcal{A}} \nc{\cca}{\mathcal{C}} \nc{\gca}{\mathcal{G}}
\nc{\sca}{\mathcal{S}} \nc{\hca}{\mathcal{H}} \nc{\bca}{\mathcal{B}}
\nc{\dca}{\mathcal{D}} \nc{\val}{\operatorname{val}}

\nc{\vp}{\varphi} \nc{\ddt}{\frac{d}{dt}} \nc{\dds}{\frac{d}{ds}}
\nc{\dpar}{\frac{\partial}{\partial t}} \nc{\im}{\mathtt{i}}

\nc{\SO}{\mathrm{SO}} \nc{\Spe}{\mathrm{Sp}} \nc{\Sl}{\mathrm{SL}}
\nc{\SU}{\mathrm{SU}} \nc{\Or}{\mathrm{O}} \nc{\U}{\mathrm{U}} \nc{\Gl}{\mathrm{GL}}
\nc{\Se}{\mathrm{S}} \nc{\Cl}{\mathrm{Cl}} \nc{\Spein}{\mathrm{Spin}}
\nc{\Pin}{\mathrm{Pin}} \nc{\G}{\mathrm{GL}_n(\RR)} \nc{\g}{\mathfrak{gl}_n(\RR)}

\nc{\RR}{{\Bbb R}} \nc{\HH}{{\Bbb H}} \nc{\CC}{{\Bbb C}} \nc{\ZZ}{{\Bbb Z}}
\nc{\FF}{{\Bbb F}} \nc{\NN}{{\Bbb N}} \nc{\QQ}{{\Bbb Q}} \nc{\PP}{{\Bbb P}}

\nc{\vs}{\vspace{.2cm}} \nc{\vsp}{\vspace{1cm}} \nc{\ip}{\langle\cdot,\cdot\rangle}
\nc{\ipp}{(\cdot,\cdot)} \nc{\la}{\langle} \nc{\ra}{\rangle} \nc{\unm}{\tfrac{1}{2}}
\nc{\unc}{\tfrac{1}{4}} \nc{\und}{\tfrac{1}{16}} \nc{\no}{\vs\noindent}
\nc{\lamkn}{\Lambda^2(\RR^{q+n})^*\otimes\RR^{q+n}} \nc{\lamn}{\Lambda^2(\RR^n)^*\otimes\RR^n} \nc{\lamp}{\Lambda^2\pg^*\otimes\pg}
\nc{\lamg}{\Lambda^2\ggo^*\otimes\ggo} \nc{\lamngo}{\Lambda^2\ngo^*\otimes\ngo}
\nc{\tangz}{{\rm T}^{\rm Zar}} \nc{\mum}{/\!\!/} \nc{\kir}{/\!\!/\!\!/}
\nc{\Ri}{\tfrac{4\Ric_{\mu}}{||\mu||^2}} \nc{\ds}{\displaystyle}
\nc{\ben}{\begin{enumerate}} \nc{\een}{\end{enumerate}} \nc{\f}{\frac}
\nc{\lb}{[\cdot,\cdot]} \nc{\isn}{\tfrac{1}{||v||^2}}
\nc{\gkp}{(\ggo=\kg\oplus\pg,\ip)} \nc{\ukh}{(\ug=\kg\oplus\hg,\ip)}

\nc{\Hess}{\operatorname{Hess}} \nc{\ad}{\operatorname{ad}}
\nc{\Ad}{\operatorname{Ad}} \nc{\rank}{\operatorname{rank}}
\nc{\Irr}{\operatorname{Irr}} \nc{\End}{\operatorname{End}}
\nc{\Aut}{\operatorname{Aut}} \nc{\Inn}{\operatorname{Inn}}
\nc{\Der}{\operatorname{Der}} \nc{\Ker}{\operatorname{Ker}}
\nc{\Iso}{\operatorname{I}} \nc{\Diff}{\operatorname{Diff}}
\nc{\Lie}{\operatorname{Lie}} \nc{\tr}{\operatorname{tr}} \nc{\dif}{\operatorname{d}}
\nc{\sen}{\operatorname{sen}} \nc{\modu}{\operatorname{mod}}
\nc{\Riem}{\operatorname{Rm}} \nc{\Ricci}{\operatorname{Ric}}
\nc{\sym}{\operatorname{sym}} \nc{\symac}{\operatorname{sym^{ac}}}
\nc{\symc}{\operatorname{sym^{c}}} \nc{\scalar}{\operatorname{R}}
\nc{\grad}{\operatorname{grad}} \nc{\ricci}{\operatorname{Rc}}
\nc{\nr}{\operatorname{nr}} \nc{\riccic}{\operatorname{ric^{c}}}
\nc{\riccig}{\operatorname{ric^{\gamma}}} \nc{\Rin}{\operatorname{M}}
\nc{\Le}{\operatorname{L}}
\nc{\level}{\operatorname{level}} \nc{\rad}{\operatorname{rad}}
\nc{\abel}{\operatorname{ab}} \nc{\CH}{\operatorname{CH}}
\nc{\mcc}{\operatorname{mcc}} \nc{\Adj}{\operatorname{Adj}}
\nc{\Order}{\operatorname{O}} \nc{\mm}{\operatorname{M}}
\nc{\inj}{\operatorname{inj}} \nc{\proy}{\operatorname{pr}}
\nc{\vol}{\operatorname{vol}}

\theoremstyle{plain}
\newtheorem{theorem}{Theorem}[section]
\newtheorem{proposition}[theorem]{Proposition}

\newtheorem{lemma}[theorem]{Lemma}

\theoremstyle{definition}
\newtheorem{definition}[theorem]{Definition}

\theoremstyle{remark}
\newtheorem{remark}[theorem]{Remark}

\newtheorem{example}[theorem]{Example}

\title{Ricci flow of homogeneous manifolds}

\author{Jorge Lauret}

\address{FaMAF and CIEM, Universidad Nacional de C\'ordoba, C\'ordoba, Argentina}
\email{lauret@famaf.unc.edu.ar}

\thanks{This research was partially supported by grants from CONICET (Argentina)
and SeCyT (Universidad Nacional de C\'ordoba)}

\begin{document}

\maketitle

\begin{abstract}
We present in this paper a general approach to study the Ricci flow on homogeneous manifolds.  Our main tool is a dynamical system defined on a subset $\hca_{q,n}$ of the variety of $(q+n)$-dimensional Lie algebras, parameterizing the space of all simply connected homogeneous spaces of dimension $n$ with a $q$-dimensional isotropy, which is proved to be equivalent in a precise sense to the Ricci flow.  The approach is useful to better visualize the possible (nonflat) pointed limits of Ricci flow solutions, under diverse rescalings, as well as to determine the type of the possible singularities.  Ancient solutions arise naturally from the qualitative analysis of the evolution equation.

We develop two examples in detail: a $2$-parameter subspace of $\hca_{1,3}$ reaching most of $3$-dimensional geometries, and a $2$-parameter family in $\hca_{0,n}$ of left-invariant metrics on $n$-dimensional compact and non-compact semisimple Lie groups.
\end{abstract}

\tableofcontents

\section{Introduction}

We present in this paper a general approach to study the Ricci flow on homogeneous manifolds.  Our main tool is a dynamical system defined on a space $\hca_{q,n}$ parameterizing the space of all simply connected homogeneous spaces of dimension $n$ with a $q$-dimensional isotropy, which is proved to be equivalent in a precise sense to the Ricci flow.  Among some other advantages, this setting helps us to better visualize the possible (nonflat) pointed limits of Ricci flow solutions, under diverse normalizations, as well as to find ancient solutions and determine the type of the singularities.  $\hca_{q,n}$ is a subset of the variety of $(q+n)$-dimensional Lie algebras.

Given a homogeneous Riemannian manifold $(M,g_0)$, consider the unique homogeneous Ricci flow solution $g(t)$ starting at $g_0$, that is,
\begin{equation}\label{RFint}
\dpar g(t)=-2\ricci(g(t)),\qquad g(0)=g_0.
\end{equation}
As the Ricci flow preserves isometries, any transitive Lie group $G\subset\Iso(M,g_0)$ provides, for all $t$, a presentation as a homogeneous space
$$
(M,g(t))=(G/K,g_{\ip_t}), \qquad \mbox{with the same reductive decomposition} \quad \ggo=\kg\oplus\pg.
$$
We denote by $g_{\ip_t}$ the $G$-invariant metric on $G/K$ determined by its value at the origin $\ip_t=g_{\ip_t}(o)$, which is an $\Ad(K)$-invariant inner product on $\pg=T_oG/K$.  The family $\ip_t$ is the solution to the ODE
\begin{equation}\label{RFipint}
\ddt \ip_t=-2\ricci(\ip_t), \qquad\mbox{where}\quad \ricci(\ip_t)=\ricci(g(t))(o),
\end{equation}
and thus the maximal interval of time where $g(t)$ exists is of the form $(T_-,T_+)$ for some $-\infty\leq T_-<0<T_+\leq\infty$.

This approach seems not to be quite the appropriate one to study certain questions on the Ricci flow of homogeneous manifolds, as it confines ourselves to the universe of $G$-invariant metrics on $M$.  For instance, the possible limits of normalized Ricci flow solutions $c_tg(t)$, as $t\to T_{\pm}$, that we may directly get for different rescalings $c_t>0$ must all necessarily be $G$-invariant metrics of the form $g_{\ip_{\pm}}$ for some $\ip_{\pm}=\lim c_t\ip_t$.  The reason for a family $c_t\ip_t$ to diverge as $t\to T_{\pm}$ can therefore be that it is actually converging to a $G_\pm$-invariant metric on some other homogeneous space $G_\pm/K_\pm$, which might be even non-homeomorphic to $M$.  The notion of convergence we mainly use in this paper is the so called pointed or Cheeger-Gromov convergence.

In order to avoid this restriction, we propose the following point of view, which can be roughly described as evolving the algebraic side of a homogeneous space rather than its metric.

Given $(M,g_0)=(G/K,g_{\ip_0})$ as above, we consider for a family $\lb_t\in\lamg$ of Lie brackets the evolution equation, called the {\it bracket flow} (see Section \ref{lbflow}), defined by
\begin{equation}\label{BFint}
\ddt\lb_t=-\pi\left(\left[\begin{smallmatrix} 0&0\\ 0&\Ricci_t
\end{smallmatrix}\right]\right)\lb_t, \qquad \lb_0=\;\mbox{Lie bracket of}\; G,
\end{equation}
where $\pi$ is the canonical representation of $\glg(\ggo)$ on $\lamg$ given by
\begin{equation}\label{actiongint}
\pi(A)\lb=A\lb-[A\cdot,\cdot]-[\cdot,A\cdot] =\dds|_0 e^{sA}[e^{-sA}\cdot,e^{-sA}\cdot],
\end{equation}
and $\Ricci_t$ denotes the Ricci operator of a homogeneous space $(G_t/K_t,g_{\lb_t})$ which is determined by $\lb_t$ (the space $\hca_{q,n}$ mentioned above turns out to be bracket flow invariant).  Here $G_t$ is the simply connected Lie group with Lie algebra $(\ggo,\lb_t)$, the reductive decomposition is always  $\ggo=\kg\oplus\pg$ and $g_{\lb_t}$ is the $G_t$-invariant metric on $G_t/K_t$ such that $g_{\lb_t}(o)=\ip_0$ (see Section \ref{varhs}).

We first prove (see Theorem \ref{eqfl}) that there exists a time-dependent family of diffeomorphisms $\vp(t):M=G/K\longrightarrow G_t/K_t$ such that
$$
g(t)=\vp(t)^*g_{\lb_t}, \qquad\forall t\in (T_-,T_+).
$$
Moreover, $\vp(t)$ can be chosen as the equivariant diffeomorphism determined by a Lie group isomorphism between $G$ and $G_t$ with derivative $\tilde{h}(t):=\left[\begin{smallmatrix} I&0\\ 0&h(t) \end{smallmatrix}\right]:\ggo\longrightarrow\ggo$ (i.e. $d\vp(t)|_o=\tilde{h}(t)|_{\pg}=h(t)$), and the curve $h(t)\in\Gl(\pg)$ satisfies
$$
\begin{array}{c}
\ddt h=-h\Ricci(\ip_t), \qquad \ddt h=-\Ricci_t h, \qquad h(0)=I, \\ \\
 \ip_t=\la h\cdot,h\cdot\ra_0, \qquad \lb_t =\tilde{h}[\tilde{h}^{-1}\cdot,\tilde{h}^{-1}\cdot]_0.
\end{array}
$$
In particular, $\lb_t$ is isomorphic to $\lb_0$ for all $t$ and both the Lie bracket on $\kg$ and the isotropy representation remain constant in time (i.e. $\lb_t|_{\kg\times\ggo}\equiv\lb_0$), so that only $\lb_t|_{\pg\times\pg}$ is really evolving.

The Ricci flow $g(t)$ therefore differs from the bracket flow $g_{\lb_t}$ only by pullback by time-dependent diffeomorphisms, and hence the behavior
of the curvature is exactly the same along any of them.  As the maximal interval of time where a solution exists also coincides for both flows, it follows that the types of the singularities $T_{\pm}$ of the Ricci flow solution $g(t)$ can be determined by analyzing the bracket flow solution $\lb_t$.  Furthermore, one can for instance find ancient solutions to the Ricci flow (i.e. $T_-=-\infty$) by just showing that $\lb_t$ (or its derivative) remains bounded in the vector space $\lamg$ backward in time.

The right rescaling here is given by
$$
c\cdot\lb|_{\kg\times\ggo}=\lb, \qquad c\cdot\lb|_{\pg\times\pg}=c^2\lb^{\kg}+c\lb^{\pg},
$$
where the superscripts denote the components on $\kg$ and $\pg$ of $\lb|_{\pg\times\pg}$, respectively (see Section \ref{varhs}).  Concerning the limit behavior, if $c_t\cdot\lb_t\to\lb_{\pm}$, as $t\to T_\pm$, then we can apply the results given in \cite{spacehm} translating this convergence of Lie brackets into more geometric notions of convergence, including the pointed topology  (see Section \ref{convergence}).  The new ingredient now is that the limit Lie bracket $\lb_{\pm}$ may be non-isomorphic to $\lb_t$, providing an explicit limit $(G_{\pm}/K_{\pm},g_{\lb_{\pm}})$ with a different algebraic presentation and which can often be non-diffeomorphic and even non-homeomorphic to $M$.

This approach has been applied to the study of the Ricci flow on nilmanifolds in \cite{nilricciflow}.  We also refer to \cite{Gzh,Pyn} for other studies of the Ricci flow on nilmanifolds via the bracket flow (\ref{BFint}).  In the case of $3$-dimensional unimodular Lie groups, a global picture of the qualitative behavior of the Ricci flow is given in \cite{GlcPyn} by using the approach proposed in this paper: to vary Lie brackets instead of inner products.  In \cite{Arr}, the bracket flow is applied to study the Ricci flow for homogeneous manifolds of dimension $4$, and in \cite{homRS} to study homogeneous Ricci solitons.  On the other hand, since the pioneer work \cite{IsnJck}, the Ricci flow for homogeneous manifolds has extensively been studied from the standard point of view in dimension $3$ (see \cite{KnpMcl,Glc,Ltt}, and for the backward case see \cite{CaoSlf}) and dimension $4$ (see \cite{IsnJckLu,Ltt}), on Lie groups (see \cite{nicebasis}), on flag manifolds (see \cite{GrmMrt, AnsChr}) and in particular on the $12$-dimensional example $\Spe(3)/\Spe(1)\times\Spe(1)\times\Spe(1)$ (see \cite{BhmWlk}), where the Ricci flow evolves certain positively curved metrics into metrics with mixed Ricci curvature.

We give, in Propositions \ref{eqs} and \ref{eqsrn}, the evolution equations the different parts of the Ricci curvature obey along the Ricci flow and its possible normalizations, which provide a very useful tool in the qualitative analysis of the bracket flow solutions.

In order to illustrate all the aspects of the approach described above, we develop two examples in detail.  In Section \ref{exdim3}, we consider a $2$-parameter subspace of $\hca_{1,3}$, which is invariant by the bracket flow and reaches the $3$-dimensional geometries $\RR^3$, Nil, $S^2\times\RR$, $H^2\times\RR$ and some homogeneous spaces isometric to left-invariant invariant metrics on $S^3$ and $\widetilde{\Sl}_2(\RR)$.  A particularly interesting feature in this example is the pointed convergence of the solution on $\widetilde{\Sl}_2(\RR)$ normalized by scalar curvature $R\equiv -\unm$ towards $H^2\times\RR$, which can never appear on $\hca_{0,3}$, the space of all $3$-dimensional Lie algebras, as $H^2\times\RR$ is not a unimodular Lie group and this is a closed condition on $\hca_{0,3}$.  Secondly, a $2$-parameter family in $\hca_{0,n}$ of left-invariant metrics on $n$-dimensional semisimple Lie groups is studied in Section \ref{exss}.  As a result we find non-Einstein ancient Ricci flow solutions on each compact simple Lie group different from $\Spe(2k+1)$, and show that the Einstein metric which is not bi-invariant represents an unstable point in both the volume and scalar curvature normalized bracket flow systems.

\vs \noindent {\it Acknowledgements.} The author is very grateful to Ramiro Lafuente for his guidance in making the figures in this paper, and to him and Mauro Subils for very helpful comments.

\section{The space of homogeneous manifolds}\label{hm}

Our aim in this section is to describe a framework developed in \cite{spacehm} which allows us to work on the `space of homogeneous manifolds', by parameterizing  the set of all homogeneous spaces of dimension $n$ and isotropy dimension $q$ by a subset $\hca_{q,n}$ of the variety of $(q+n)$-dimensional Lie algebras.

Let $(M,g)$ be a connected {\it homogeneous} Riemannian manifold, i.e. its isometry group $\Iso(M,g)$ acts transitively on $M$.  Then each closed Lie subgroup $G\subset\Iso(M,g)$ acting transitively on $M$ (which can be assumed to be connected) gives rise to a
presentation of $(M,g)$ as a homogeneous space $(G/K,g_{\ip})$, where
$K$ is the isotropy subgroup of $G$ at some point $o\in M$.  Since $K$ turns to be compact, there always
exists an $\Ad(K)$-invariant direct sum decomposition $\ggo=\kg\oplus\pg$, where $\ggo$ and $\kg$ are respectively the Lie algebras of $G$ and $K$.  This is called a {\it reductive decomposition} and is in general non-unique.  Thus $\pg$ can be naturally identified with the tangent space $\pg\equiv T_oM=T_oG/K$, by taking the value at the origin $o=eK$ of the Killing vector fields
corresponding to elements of $\pg$ (i.e. $X_o=\ddt|_0\exp{tX}(o)$).  We denote by $g_{\ip}$ the $G$-invariant metric on $G/K$ determined by $\ip:=g(p)$, the $\Ad(K)$-invariant inner product on $\pg$ defined by $g$.

Any homogeneous space $G/K$ will be assumed to be {\it almost-effective}, i.e. $K$ contains no non-discrete normal subgroup of
$G$, or equivalently, the normal subgroup $\{ g\in G:ghK=hK, \;\forall h\in G\}$ is discrete.

\subsection{Varying Lie brackets viewpoint}\label{varhs}
Let us fix for the rest of the paper a $(q+n)$-dimensional real vector space $\ggo$ together with a direct sum decomposition
\begin{equation}\label{fixdec}
\ggo=\kg\oplus\pg, \qquad \dim{\kg}=q, \qquad \dim{\pg=n},
\end{equation}
and an inner product $\ip$ on $\pg$.  We consider the space of all skew-symmetric algebras (or brackets) of dimension $q+n$, which is
parameterized by the vector space
$$
V_{q+n}:=\{\mu:\ggo\times\ggo\longrightarrow\ggo : \mu\; \mbox{bilinear and
skew-symmetric}\},
$$
and we set
$$
V_{n}:=\{\mu:\pg\times\pg\longrightarrow\pg : \mu\; \mbox{bilinear and
skew-symmetric}\}.
$$
For any $X\in\ggo$, we denote left multiplication (or adjoint action) as usual by $\ad_{\mu}{X}(Y)=\mu(X,Y)$ for all $Y\in\ggo$.

\begin{definition}\label{hqn} The subset $\hca_{q,n}\subset V_{q+n}$ consists of the brackets $\mu\in V_{q+n}$ such that:
\begin{itemize}
\item [(h1)]  $\mu$ satisfies the Jacobi condition, $\mu(\kg,\kg)\subset\kg$ and $\mu(\kg,\pg)\subset\pg$.

\item[(h2)] If $G_\mu$ denotes the simply connected Lie group with Lie algebra $(\ggo,\mu)$ and $K_\mu$ is the connected Lie subgroup of $G_\mu$ with Lie algebra $\kg$, then $K_\mu$ is closed in $G_\mu$.

\item[(h3)] $\ip$ is $\ad_{\mu}{\kg}$-invariant (i.e. $(\ad_{\mu}{Z}|_{\pg})^t=-\ad_{\mu}{Z}|_{\pg}$ for all $Z\in\kg$).

\item[(h4)] $\{ Z\in\kg:\mu(Z,\pg)=0\}=0$.
\end{itemize}
\end{definition}

Each $\mu\in\hca_{q,n}$ defines a unique simply connected homogeneous space,
\begin{equation}\label{hsmu}
\mu\in\hca_{q,n}\rightsquigarrow\left(G_{\mu}/K_{\mu},g_\mu\right),
\end{equation}
with $\Ad(K_{\mu})$-invariant decomposition $\ggo=\kg\oplus\pg$ and $g_\mu(o_\mu)=\ip$, where $o_\mu:=e_\mu K_\mu$ is the origin of $G_\mu/K_\mu$ and $e_\mu\in G_\mu$ is the identity element.  It is almost-effective by (h4), and it follows from (h3) that $\ip$ is $\Ad(K_{\mu})$-invariant as $K_{\mu}$ is connected.  We note that any $n$-dimensional simply connected homogeneous Riemannian space $(G/K,g_{\ip})$ which is almost-effective can be identified with some $\mu\in\hca_{q,n}$, where $q=\dim{K}$.  Indeed, $G$ can be assumed to be simply connected without losing almost-effectiveness, and we can identify any $\Ad(K)$-invariant decomposition with $\ggo=\kg\oplus\pg$. In this way, $\mu$ will be precisely the Lie bracket of $\ggo$.

We also fix from now on a basis $\{ Z_1,\dots,Z_q\}$ of $\kg$ and an orthonormal basis $\{ X_1,\dots,X_n\}$ of $\pg$ (see (\ref{fixdec})) and use them to identify the groups $\Gl(\ggo)$, $\Gl(\kg)$, $\Gl(\pg)$ and $\Or(\pg,\ip)$, with $\Gl_{q+n}(\RR)$, $\Gl_q(\RR)$, $\Gl_n(\RR)$ and $\Or(n)$, respectively.

There is a natural linear action of $\Gl_{q+n}(\RR)$ on $V_{q+n}$ given by
\begin{equation}\label{action}
h\cdot\mu(X,Y)=h\mu(h^{-1}X,h^{-1}Y), \qquad X,Y\in\ggo, \quad h\in\Gl_{q+n}(\RR),\quad \mu\in V_{q+n}.
\end{equation}
The following result gives a useful geometric meaning to the action of some subgroups or subsets of $\Gl_{q+n}(\RR)$ on $\hca_{q,n}$.

\begin{proposition}\label{const}\cite{spacehm}
If $\mu\in\hca_{q,n}$, then $h\cdot\mu\in\hca_{q,n}$ for any $h\in\Gl_{q+n}(\RR)$ of the form
\begin{equation}\label{formh}
h:=\left[\begin{smallmatrix} h_q&0\\ 0&h_n
\end{smallmatrix}\right]\in\Gl_{q+n}(\RR), \quad h_q\in\Gl_q(\RR), \quad h_n\in\Gl_n(\RR),
\end{equation}
such that
\begin{equation}\label{adkh}
[h_n^th_n,\ad_{\mu}{\kg}|_{\pg}]=0.
\end{equation}
In that case,
\begin{itemize}
\item[(i)] $G_{h\cdot\mu}/K_{h\cdot\mu}$ and $G_{\mu}/K_{\mu}$ are equivariantly diffeomorphic.

\item[(ii)] $\left(G_{h\cdot\mu}/K_{h\cdot\mu},g_{h\cdot\mu}\right)$ is equivariantly isometric to $\left(G_{\mu}/K_{\mu},g_{\la h_n\cdot,h_n\cdot\ra}\right)$.
\end{itemize}
\end{proposition}

It follows from part (ii) that the metrics $g_\mu$ and $g_{h\cdot\mu}$ have the same volume element if $\det{h_n}=1$, and that the subset
$$
\left\{ h\cdot\mu:h_q=I,\, h_n\,\mbox{satisfies (\ref{adkh})}\right\}\subset\hca_{q,n},
$$
parameterizes the set of all $G_\mu$-invariant metrics on $G_\mu/K_\mu$.  Also, by setting $h_q=I$, $h_n=\frac{1}{c}I$, $c\ne 0$, we get the rescaled $G_\mu$-invariant metric $\frac{1}{c^2}g_{\ip}$ on $G\mu/K_\mu$, which is isometric according to (ii) and (\ref{action}) to the element of $\hca_{q,n}$ denoted by $c\cdot\mu$ and defined by
\begin{equation}\label{scmu}
c\cdot\mu|_{\kg\times\kg}=\mu, \qquad c\cdot\mu|_{\kg\times\pg}=\mu, \qquad c\cdot\mu|_{\pg\times\pg}=c^2\mu_{\kg}+c\mu_{\pg},
\end{equation}
where the subscripts denote the $\kg$ and $\pg$-components of $\mu|_{\pg\times\pg}$ given by
\begin{equation}\label{decmu}
\mu(X,Y)=\mu_{\kg}(X,Y)+\mu_{\pg}(X,Y), \qquad \mu_{\kg}(X,Y)\in\kg, \quad \mu_{\pg}(X,Y)\in\pg, \qquad\forall X,Y\in\pg.
\end{equation}
The $\RR^*$-action on $\hca_{q,n}$, $\mu\mapsto c\cdot\mu$, can therefore be considered as a geometric rescaling of the homogeneous space $(G_\mu/K_\mu,g_\mu)$.

\subsection{Convergence}\label{convergence}
We now survey some definitions and results given in \cite{spacehm} about convergence of homogeneous manifolds. In order to define some notions of convergence, of a sequence $(M_k,g_k)$ of homogeneous manifolds, to a homogeneous manifold $(M,g)$, we start by requiring the existence of a sequence $\Omega_k\subset M$ of open neighborhoods of a basepoint $p\in M$ together with embeddings $\phi_k:\Omega_k\longrightarrow M_k$ such that $\phi_k^*g_k\to g$ {\it smoothly} as $k\to\infty$ (i.e. the tensor $\phi_k^*g_k-g$ and its covariant derivatives of all orders each converge uniformly to zero on compact subsets of $M$ eventually contained in all $\Omega_k$).  We set $p_k:=\phi_k(p)\in M_k$.  According to the different conditions one may require on the size of the $\Omega_k$'s, we have the following notions of convergence in increasing degree of strength:

\begin{itemize}
\item {\it infinitesimal}: no condition on $\Omega_k$, it may even happen that $\bigcap\Omega_k=\{ p\}$ (i.e. only the germs of the metrics at $p$ are involved).  Nevertheless, the sequence $(M_k,g_k)$ has necessarily {\it bounded geometry} by homogeneity (i.e. for all $r>0$ and $j\in\ZZ_{\geq 0}$,
$\sup\limits_k\;\sup\limits_{B_{g_k}(p_k,r)} \|\nabla_{g_k}^j\Riem(g_k)\|_{g_k}<\infty$).

\item {\it local}: $\Omega_k$ stabilizes, i.e. there is a nonempty open subset $\Omega\subset\Omega_k$ for every $k$ sufficiently large.  A positive lower bound for the injectivity radii $\inj(M_k,g_k,p_k)$ therefore holds, which is often called the {\it non-collapsing} condition.

\item {\it pointed or Cheeger-Gromov}: $\Omega_k$ exhausts $M$, i.e. $\Omega_k$ contains any compact subset of $M$ for $k$ sufficiently large.  We note that in the homogeneous case, the location of the basepoints $p_k\in M_k$ and $p\in M$ play no role, neither in the pointed convergence nor in the bounds considered in the items above, in the sense that we can change all of them by any other sequence $q_k\in M_k$ and $q\in M$ and use homogeneity.  However, topology issues start to come in at this level of convergence.

\item {\it smooth (up to pull-back by diffeomorphisms)}: $\Omega_k=M$ and $\phi_k:M\longrightarrow M_k$ is a diffeomorphism for all $k$.  This necessarily holds when $M$ is compact.  Recall that if $M$ is noncompact, then $\phi_k^*g_k$ converges smoothly to $g$ uniformly on compact sets in $M$.
\end{itemize}

It follows at once from the definitions that these notions of convergence of homogeneous manifolds satisfy:

\begin{center}
smooth $\quad\Rightarrow\quad$ pointed $\quad\Rightarrow\quad$ local $\quad\Rightarrow\quad$ infinitesimal.
\end{center}

None of the converse assertions hold (see \cite{spacehm}).  However, it is worth noticing that local convergence implies bounded geometry and non-collapsing for the sequence, and thus there must exist a pointed convergent subsequence by the compactness theorem.

We may also consider convergence of the algebraic side of homogeneous manifolds.  Recall from Section \ref{varhs} the space $\hca_{q,n}$ of Lie algebras parameterizing the set of all $n$-dimensional simply connected homogeneous spaces with $q$-dimensional isotropy, which inherits the usual vector space topology from $V_{q+n}$.  We shall always denote by $\mu_k\to\lambda$ the convergence in $\hca_{q,n}$ relative to such topology, which turns out to be essentially equivalent to infinitesimal convergence of homogeneous manifolds.

\begin{theorem}\label{convmu}\cite[Theorem 6.12]{spacehm}
Let $\mu_k$ be a sequence and $\lambda$ an element in $\hca_{q,n}$.
\begin{itemize}
\item[(i)] If $\mu_k\to\lambda$ in $\hca_{q,n}$ (usual vector space topology), then $(G_{\mu_k}/K_{\mu_k},g_{\mu_k})$ infinitesimally converges to $(G_{\lambda}/K_{\lambda},g_{\lambda})$.

\item[(ii)] If $(G_{\mu_k}/K_{\mu_k},g_{\mu_k})$ infinitesimally converges to $(G_{\lambda}/K_{\lambda},g_{\lambda})$, then
    $$
    (\mu_k)_{\pg}\to\lambda_{\pg},
    $$
    where the subscript $\pg$ is defined as in {\rm (\ref{decmu})}.
\end{itemize}
\end{theorem}

Both the converse assertions to (i) and (ii) in this theorem can in general be false (see \cite[Remark 6.13]{spacehm}).  As some sequences of Aloff-Walach spaces show (see \cite[Example 6.6]{spacehm}), in order to get the stronger local convergence from the usual convergence of brackets $\mu_k\to\lambda$, and consequently pointed subconvergence, it is necessary (and also sufficient) to require an `algebraic' non-collapsing type condition.
The {\it Lie injectivity radius} of $\left(G_\mu/K_\mu,g_\mu\right)$, $\mu\in\hca_{q,n}$, is the largest $r_\mu>0$ such that
$$
\pi_\mu\circ\exp_\mu: B(0,r_\mu)\longrightarrow G_\mu/K_\mu,
$$
is a diffeomorphism onto its image, where $\exp_\mu:\ggo\longrightarrow G_{\mu}$ is the Lie exponential map, $\pi_\mu:G_\mu\longrightarrow G_\mu/K_\mu$ is the usual quotient map and $B(0,r_\mu)$ denotes the euclidean ball of radius $r_\mu$ in $(\pg,\ip)$.

\begin{theorem}\label{convmu2}\cite[Theorem 6.14]{spacehm}
Let $\mu_k$ be a sequence such that $\mu_k\to\lambda$ in $\hca_{q,n}$, as $k\to\infty$, and assume that $\inf\limits_k r_{\mu_k}>0$.  Then,
\begin{itemize}
\item[(i)] $(G_{\mu_k}/K_{\mu_k},g_{\mu_k})$ locally converges to $(G_{\lambda}/K_{\lambda},g_{\lambda})$.

\item[(ii)] There exists a subsequence of $(G_{\mu_k}/K_{\mu_k},g_{\mu_k})$ which converges in the pointed sense to a homogeneous manifold locally isometric to $(G_{\lambda}/K_{\lambda},g_{\lambda})$.

\item[(iii)] $(G_{\mu_k}/K_{\mu_k},g_{\mu_k})$ converges in the pointed sense to $(G_{\lambda}/K_{\lambda},g_{\lambda})$ if $G_\lambda/K_\lambda$ is compact.
\end{itemize}
\end{theorem}

The limit for the pointed subconvergence stated in the above theorem may not be unique, as certain sequence of alternating left-invariant metrics on $S^3$ (Berger spheres) and $\widetilde{\Sl_2}(\RR)$ shows (see \cite[Example 6.17]{spacehm}).

\begin{example}\label{lieinj}
We take the basis of the Lie subalgebra $\RR I\oplus\sug(2)\subset\glg_2(\CC)$ given by
$$
Y_0=\left[\begin{matrix} 1&0 \\ 0&1 \end{matrix}\right], \quad
Y_1=\left[\begin{matrix} 0&-1 \\ 1&0 \end{matrix}\right], \quad
Y_2=\left[\begin{matrix} \im&0 \\ 0&-\im \end{matrix}\right], \quad
Y_3=\left[\begin{matrix} 0&\im \\ \im&0 \end{matrix}\right],
$$
whose nonzero Lie brackets are
$$
[Y_1,Y_2]=2Y_3, \qquad [Y_1,Y_3]=-2Y_2, \qquad [Y_2,Y_3]=2Y_1.
$$
For $a,b\in\RR$, $b>0$, the new basis
$$
Z_1=-\tfrac{a}{b}Y_0+\unm Y_1, \quad X_1=Y_0, \quad X_2=\tfrac{\sqrt{b}}{2}Y_2, \quad X_3=\tfrac{\sqrt{b}}{2}Y_3,
$$
satisfies
\begin{equation}\label{lbab}
[X_3,Z_1]=X_2, \qquad [Z_1,X_2]=X_3, \qquad [X_2,X_3]=aX_1+bZ_1.
\end{equation}
Let $\mu=\mu_{a,b}$ denote the element of $\hca_{1,3}$ defined as in (\ref{lbab}) (see Section \ref{varhs}).  Then $(\ggo,\mu)$ is isomorphic to $\RR\oplus\sug(2)$ and $(G_\mu/K_\mu,g_\mu)$ is isometric to a left-invariant metric on $S^3$ for any $b>0$.  In order to apply Theorem \ref{convmu2}, we need to know a lower bound for the Lie injectivity radius $r_\mu$.  We have that $K_\mu=\exp_\mu(\RR Z_1)$, and thus
\begin{equation}\label{expeq}
\pi_\mu(\exp_\mu(r_1X_1+r_2X_2+r_3X_3))=\pi_\mu(\exp_\mu(s_1X_1+s_2X_2+s_3X_3))
\end{equation}
if and only if
$$
e^{r_1}\exp\left(\tfrac{\im\sqrt{b}}{2}\left[\begin{smallmatrix} r_2&r_3 \\ r_3&-r_2 \end{smallmatrix}\right]\right)
=e^{s_1}\exp\left(\tfrac{\im\sqrt{b}}{2}\left[\begin{smallmatrix} s_2&s_3 \\ s_3&-s_2 \end{smallmatrix}\right]\right)e^{-at/b} \left[\begin{smallmatrix} \cos{t/2}&-\sin{t/2} \\ \sin{t/2}&\cos{t/2} \end{smallmatrix}\right],
$$
for some $t\in\RR$, where $\exp:\glg_2(\CC)\longrightarrow\Gl_2(\CC)$ is the usual exponential map.  By using uniqueness of the polar decomposition and the formula
$$
\exp\left(\im\left[\begin{smallmatrix} x&y \\ y&-x \end{smallmatrix}\right]\right)=
\tfrac{1}{N}\left[\begin{smallmatrix} x\sin{N}+N\cos{N}& y\sin{N} \\ y\sin{N}& -x\sin{N}+N\cos{N}\end{smallmatrix}\right],\qquad N:=\sqrt{x^2+y^2},
$$
it is easy to see that if (\ref{expeq}) holds then
$$
s_1-r_1=4k\pi\tfrac{a}{b}, \qquad \tfrac{\sqrt{b}}{2}\sqrt{r_2^2+r_3^2}=\pm\tfrac{\sqrt{b}}{2}\sqrt{s_2^2+s_3^2}+2j\pi , \qquad\mbox{for some}\; k,j\in\ZZ.
$$
We deduce that $\pi_\mu\circ\exp_\mu$ is injective on $B(0,r)$ for any $b>0$, where $r=\min\{2\pi\tfrac{|a|}{b}, 2\pi\tfrac{1}{\sqrt{b}}\}$.  On the other hand, it is well-known that
\begin{equation}\label{expdif}
\exp_\mu:B(0,\tfrac{\pi}{\|\mu\|})\longrightarrow G_\mu
\end{equation}
is always a diffeomorphism onto its image (see e.g. \cite[Lemma 6.19]{spacehm}), and consequently
$$
r_\mu\geq \min\left\{2\pi\tfrac{|a|}{b},\; 2\pi\tfrac{1}{\sqrt{b}},\; \tfrac{\pi}{\sqrt{2(a^2+b^2+2)}}\right\}.
$$
The analysis for $b<0$ is quite different.  Consider the basis $\{ Y_0,Y_1,\im Y_2,\im Y_3\}$ of the Lie algebra $\glg_2(\RR)=\RR I\oplus\slg_2(\RR)$.  The new basis
$$
Z_1=-\tfrac{a}{b}Y_0+\unm Y_1, \quad X_1=Y_0, \quad X_2=\tfrac{\sqrt{-b}}{2}\im Y_2, \quad X_3=\tfrac{\sqrt{-b}}{2}\im Y_3,
$$
gives the same Lie bracket as in (\ref{lbab}), that is, we get $\mu=\mu_{a,b}$ for $b<0$.  Thus $(G_\mu/K_\mu,g_\mu)$ is isometric to a left-invariant metric on $\widetilde{\Sl}_2(\RR)$ for any $b<0$, the simply connected cover of $\Sl_2(\RR)$.  If $\widetilde{\exp}:\slg_2(\RR)\longrightarrow\widetilde{\Sl}_2(\RR)$ is the exponential map, then
\begin{align*}
\exp_\mu(r_1X_1+r_2X_2+r_3X_3)=& \exp_\mu(r_1X_1)\exp_\mu(r_2X_2+r_3X_3) \\
=&\exp_\mu(r_1Y_0)\exp_\mu(r_2\tfrac{\sqrt{-b}}{2}\im Y_2+r_3\tfrac{\sqrt{-b}}{2}\im Y_3) \\
=& \left(r_1,\widetilde{\exp}(r_2\tfrac{\sqrt{-b}}{2}\im Y_2+r_3\tfrac{\sqrt{-b}}{2}\im Y_3)\right)\in\RR\times\widetilde{\Sl_2}(\RR)=G_\mu.
\end{align*}
Since
$$
\exp_\mu(tZ_1)=\left(-\tfrac{at}{b},\widetilde{\exp}(\tfrac{t}{2}Y_1)\right)\in\RR\times\widetilde{\Sl_2}(\RR)=G_\mu, \qquad\forall t\in\RR,
$$
and $(x,y,z)\mapsto\widetilde{\exp}(xY_1)\widetilde{\exp}(yY_2+zY_3)$ is a diffeomorphism between $\RR^3$ and $\widetilde{\Sl_2}(\RR)$, it follows that (\ref{expeq}) holds if and only if $t=0$ and $r_i=s_i$ for all $i=1,2,3$.  This implies that $r_\mu=\infty$ for any $b<0$.

Finally, for $b=0$, one has that $r_\mu=\infty$ as $G_\mu=K_\mu\ltimes H_3$, where $H_3$ denotes the Heisenberg Lie group, and so $\exp_\mu:\pg\longrightarrow G_\mu$ coincides with the exponential map of $H_3$, which is well-known to be a diffeomorphism.

As a consequence of the computation of $r_\mu$ given above, we obtain that if $(a_k,b_k)\to (a_\infty,b_\infty)$, then $\mu_{(a_k,b_k)}\to \mu_{(a_\infty,b_\infty)}$ produces the following types of convergence:

\begin{itemize}
\item $a_\infty\ne 0$, $b_\infty>0$: pointed, by Theorem \ref{convmu2}, (iii).

\item $b_\infty\leq 0$: local (or pointed subconvergence), by Theorem \ref{convmu2}, (i) and (ii).
\end{itemize}
See Section \ref{exdim3} for a study of the Ricci flow on these manifolds.
\end{example}

In the case of left-invariant metrics on Lie groups (i.e. $\hca_{0,n}=\lca_n$, the variety of $n$-dimensional Lie algebras), the parts of the brackets $\mu_k$ which might not be affected by an infinitesimal convergence $\tilde{g}_{\mu_k}\to\tilde{g}_\lambda$ are not present, as $q=0$, and the condition $\inf\limits_k r_{\mu_k}>0$ automatically holds when $\mu_k\to\lambda$ by (\ref{expdif}).

Theorems \ref{convmu2} and \ref{convmu} can therefore be rephrased in the case of Lie groups in a stronger way (see Theorem \ref{convmu3}).  One can say even more in some particular cases as compact or completely solvable Lie groups.  Another advantage in this case is that $\lca_n$ is closed in $V_n$, and thus any limit $\lambda$ of a sequence $\mu_k\in\lca_n$ must lie in $\lca_n$ and can always be identified with a left invariant metric on some Lie group.  On the contrary, $\hca_{q,n}$ is never closed in $V_{q+n}$ for $q>0$.

Recall that $\mu\in\lca_n$ is said to be {\it completely solvable} if all the eigenvalues of $\ad_\mu{X}$ are real for any $X$; in particular, $G_\mu$ is solvable and $\exp_\mu:\ggo\longrightarrow G_\mu$ is a diffeomorphism.  For $\mu\in\lca_n$, we denote by $(G_\mu,\ip)$ the corresponding homogeneous space $(G_\mu,g_\mu)$ attached to $\mu$ as in (\ref{hsmu}), i.e. the Lie group $G_\mu$ endowed with the left-invariant metric determined by the fixed inner product $\ip$ on $\ggo$.

\begin{theorem}\label{convmu3}\cite[Corollary 6.20]{spacehm}
Let $\mu_k$ be a sequence in $\lca_n=\hca_{0,n}$   Then the following conditions are equivalent:
\begin{itemize}
\item[(i)] $\mu_k\to\lambda$ in $\lca_n$ (usual vector space topology).

\item[(ii)] $(G_{\mu_k},\ip)$ infinitesimally converges to $(G_{\lambda},\ip)$.

\item[(iii)] $(G_{\mu_k},\ip)$ locally converges to $(G_{\lambda},\ip)$.

\item[(iv)] $(G_{\mu_k},\ip)$ converges in the pointed sense to $(G_{\lambda},\ip)$, provided $G_\lambda$ is compact or all $\mu_k$ are completely solvable.

\item[(v)] $g_{\mu_k}\to g_\lambda$ smoothly on $\RR^n\equiv\ggo$, provided all $\mu_k$ are completely solvable, where all $G_{\mu_k}$ and $G_\lambda$ are identified with $\ggo$ via the corresponding exponential maps.
\end{itemize}
In any case, there is always a subsequence of $(G_{\mu_k},\ip)$ that is convergent in the pointed sense to a homogeneous manifold locally isometric to $(G_{\lambda},\ip)$.
\end{theorem}

\subsection{Ricci curvature}\label{riccisec}
As within any class of Riemannian manifolds, in order to study the Ricci flow for homogeneous spaces, a deep understanding of their Ricci curvature is crucial.  In this section, we describe the Ricci operator of $\left(G_\mu/K_\mu,g_\mu\right)$ for any $\mu\in\hca_{q,n}$.  Recall that $\ggo=\kg\oplus\pg$ is always the  $\Ad(K_{\mu})$-invariant or reductive decomposition, $\pg=T_{o_\mu}G_\mu/K_\mu$ and $g_\mu(o_\mu)=\ip$.

There is a unique element $H_{\mu}\in\pg$
which satisfies $\la H_{\mu},X\ra=\tr{\ad_\mu{X}}$ for any $X\in\pg$. Let $B_{\mu}:\pg\longrightarrow\pg$
denote the symmetric map defined by
$$
\la B_{\mu}X,Y\ra=\tr{\ad_\mu{X}\ad_\mu{Y}},\qquad\forall X,Y\in\pg.
$$
According to \cite[7.38]{Bss}, the Ricci operator $\Ricci_\mu$ of $\left(G_\mu/K_\mu,g_\mu\right)$ is given by:
\begin{equation}\label{ricci}
\Ricci_\mu=\mm_{\mu}-\unm B_{\mu}-S(\ad_{\mu}{H_{\mu}}|_{\pg}),
\end{equation}
where
\begin{equation}\label{sym}
S:\glg_n(\RR)\longrightarrow\glg_n(\RR), \qquad S(A):=\unm(A+A^t),
\end{equation}
is the symmetric part of an operator, and $\mm_{\mu}:\pg\longrightarrow\pg$ is the
symmetric operator defined by
\begin{align}\label{R}
\la \mm_{\mu}X,Y\ra =& -\unm\sum\la \mu(X,X_i),X_j\ra\la \mu(Y,X_i),X_j\ra \\
& +\unc\sum\la \mu(X_i,X_j),X\ra\la \mu(X_i,X_j),Y\ra, \qquad\forall X,Y\in\pg, \notag
\end{align}
for any orthonormal basis $\{ X_1,\dots,X_n\}$ of $(\pg,\ip)$.

A closer inspection to formula (\ref{ricci}) for $\Ricci_{\mu}$ shows that $H_{\mu}$ only depends on
$\mu_{\pg}$, as $\mu(\kg,\pg)\subset\pg$ (see (\ref{decmu})).  Notice that if the Lie algebra $(\kg,\mu)$ is unimodular, then
$H_{\mu_{\pg}}=0$ if and only if $(\ggo,\mu)$ is unimodular.  Also, the
operator $\mm_{\mu}$ only depends on $\mu_{\pg}$ and satisfies
\begin{equation}\label{mmR}
m(\mu_{\pg})=\tfrac{4}{||\mu_{\pg}||^2} \mm_{\mu_{\pg}},
\end{equation}
where $m:V_n\longrightarrow\sym(\pg)$ is the moment map for the
natural action of $\Gl_n(\RR)$ on $V_n:=\lamp$ (see e.g. \cite{HnzSchStt} or \cite{cruzchica} for more information on real moment maps).  In other words,
$\mm_{\mu_{\pg}}$ may be alternatively defined as follows:
\begin{equation}\label{Rmm}
\tr{\mm_{\mu_{\pg}}E}=\unc\la\pi(E)\mu_{\pg},\mu_{\pg}\ra, \qquad\forall
E\in\g,
\end{equation}
where $\pi(E)\mu_{\pg}=E\mu_{\pg}(\cdot,\cdot)-\mu_{\pg}(E\cdot,\cdot)-\mu_{\pg}(\cdot,E\cdot)$ is the representation of $\glg_n(\RR)$ on $V_n$ obtained as the derivative of the action (\ref{action}) with $q=0$ (see \cite[Proposition 3.5]{minimal}).  Here $\ip$ also denotes the inner product defined on $V_n$ by
\begin{equation}\label{innV}
\la\mu,\lambda\ra= \sum\la\mu(X_i,X_j),\lambda(X_i,X_j)\ra.
\end{equation}
Concerning $B_{\mu}$, we note that this is the Killing form of the Lie algebra $(\ggo,\mu)$ restricted to $\pg\times\pg$ in terms of the inner product $\ip$, and so it does depend on the value of $\mu$ on the whole space $\ggo\times\ggo$.

A formula for the Ricci operator $\Ricci_\mu$ of $\left(G_{\mu}/K_{\mu},g_\mu\right)$ may
therefore be rewritten as
\begin{equation}\label{ricci2}
\Ricci_\mu=\mm_{\mu_{\pg}}-\unm B_{\mu}-S(\ad_{\mu_{\pg}}{H_{\mu_{\pg}}}),
\end{equation}
where $\ad_{\mu_{\pg}}{H_{\mu_{\pg}}}:\pg\longrightarrow\pg$ is defined by
$\ad_{\mu_{\pg}}{H_{\mu_{\pg}}}(X)=\mu_{\pg}(H_{\mu_{\pg}},X)$ for all $X\in\pg$.

It follows that the scalar curvature is then given by
\begin{equation}\label{scalar}
R(\mu)=-\unc\|\mu_{\pg}\|^2-\unm\tr{B_\mu}-\| H_{\mu_{\pg}}\|^2.
\end{equation}

The next two examples reach all $3$-dimensional geometries.

\begin{example}\label{ex0-3ricci}
Let $\mu=\mu_{a,b,c}$ be the Lie bracket in $\hca_{0,3}=\lca_3$ defined by
$$
\mu(X_2,X_3)=aX_1, \qquad \mu(X_3,X_1)=bX_2, \qquad \mu(X_1,X_2)=cX_3.
$$
It is easy to see that $H_{\mu}=0$,
$$
\begin{array}{c}
B_\mu=\left[\begin{smallmatrix} -2bc&&\\ &-2ac& \\ && -2ab\end{smallmatrix}\right], \quad
M_\mu=-\unm\left[\begin{smallmatrix} -a^2+b^2+c^2&&\\ &a^2-b^2+c^2& \\ &&a^2+b^2-c^2\end{smallmatrix}\right], \\ \\
\Ricci_\mu=\unm\left[\begin{smallmatrix} a^2-(b-c)^2&&\\ &b^2-(a-c)^2& \\ &&c^2-(a-b)^2\end{smallmatrix}\right],
\end{array}
$$
and $R(\mu)=-\unm(a^2+b^2+c^2)+ab+ac+bc$.  This family covers all left-invariant metric (up to isometry and scaling) on all unimodular $3$-dimensional simply connected Lie groups, which are given by $S^3$, $\widetilde{\Sl_2}(\RR)$, $E(2)$, $Sol$, $Nil$ and $\RR^3$ (see \cite{Mln}, \cite{GlcPyn} or \cite[Example 3.2]{spacehm} for a more detailed treatment of this example).
\end{example}

\begin{example}\label{ex1-3ricci}
We refer to \cite[Example 3.3]{spacehm} for more details on this example.  Consider the bracket $\mu=\mu_{a,b,c}\in\hca_{1,3}$ given by
$$
\left\{\begin{array}{lll}
\mu(X_3,Z_1)=X_2, & \mu(X_2,X_3)=aX_1+bZ, & \mu(X_3,X_1)=cX_2, \\
\mu(Z_1,X_2)=X_3, &                           & \mu(X_1,X_2)=cX_3,
\end{array}\right.
$$
where $\kg=\RR Z_1$ and $\{ X_1,X_2,X_3\}$ is an orthonormal basis of $(\pg,\ip)$.  It is easy to see that $H_{\mu}=0$,
$$
\begin{array}{c}
B_\mu=\left[\begin{smallmatrix} -2c^2&&\\ &-2(b+ac)& \\ && -2(b+ac)\end{smallmatrix}\right], \quad
M_\mu=-\unm\left[\begin{smallmatrix} 2c^2-a^2&&\\ &a^2& \\ &&a^2\end{smallmatrix}\right], \\ \\
\Ricci_\mu=\left[\begin{smallmatrix} \unm a^2&&\\ &-\unm a^2+b+ac& \\ &&-\unm a^2+b+ac\end{smallmatrix}\right], \quad
R(\mu)=-\unm a^2+2(b+ac).
\end{array}
$$
The homogeneous metrics attained by this family are
$\RR\times S^2$, $\RR\times H^2$, the flat metric on $E(2)$, some left-invariant metrics on $S^3$ (Berger spheres) and $\widetilde{\Sl_2}(\RR)$, and all the left-invariant metrics on the Heisenberg group ($Nil$) and $\RR^3$, up to isometry.
\end{example}

\section{Homogeneous Ricci flows}\label{hrf}

Let $(M,g_0)$ be a simply connected homogeneous manifold and consider a presentation of $(M,g_0)$ as a homogeneous space of the form $\left(G_{\mu_0}/K_{\mu_0},g_{\mu_0}\right)$ for some $\mu_0\in\hca_{q,n}$, with reductive decomposition $\ggo=\kg\oplus\pg$ (see Section \ref{varhs}).  Let $g(t)$ be a solution to the {\it Ricci flow}
\begin{equation}\label{RF}
\dpar g(t)=-2\ricci(g(t)),\qquad g(0)=g_0.
\end{equation}
The short time existence of a solution follows from \cite{Shi}, as $g$ is homogeneous and hence complete and of bounded curvature. Alternatively, one may require $G_{\mu_0}$-invariance of $g(t)$ for all $t$, and thus $(M,g(t))$ would have, as a homogeneous space, the presentation $\left(G_{\mu_0}/K_{\mu_0},g_{\ip_t}\right)$, where $\ip_t:=g(t)(o_{\mu_0})$ is a family of inner products on $\pg$.  The Ricci flow equation (\ref{RF}) is therefore equivalent to the ODE
\begin{equation}\label{RFip}
\ddt \ip_t=-2\ricci(\ip_t), \qquad\ip_0=\ip,
\end{equation}
where $\ricci(\ip_t):=\ricci(g(t))(o_{\mu_0})$, and hence short time existence and uniqueness of the solution in the class of $G_{\mu_0}$-invariant metrics is guaranteed.  Recall that the set of all inner products on $\pg$ is parameterized by the symmetric space $\Gl_n(\RR)/\Or(n)$ and the subset of those which are $\Ad(K_{\mu_0})$-invariant is a submanifold of it.  As the field defined by $\ricci$ is tangent to this submanifold, the solution $\ip_t$ to (\ref{RFip}) stays $\Ad(K_{\mu_0})$-invariant for all $t$.  In this way, $g(t)$ is homogeneous for all $t$, and hence the uniqueness within the set of complete and with bounded curvature metrics follows from \cite{ChnZhu}.  It is actually a simple matter to prove that such a uniqueness result, in turn, implies our assumption of $G_{\mu_0}$-invariance, as the solution must preserve any isometry of the initial metric.  The need for this circular argument is due to the fact that the uniqueness of the Ricci flow solution is still an open problem in the noncompact general case (see \cite{Chn}).

In any case, there is an interval $(a,b)\subset\RR$ such that $0\in (a,b)$ and where existence and uniqueness (within complete and with bounded curvature metrics) of the Ricci flow $g(t)$ starting at a homogeneous manifold $(M,g_0)$ hold.

\begin{remark}
One can use in the homogeneous case the existence and uniqueness of the solution $g(t)$, forward and backward from any $t\in(a,b)$, to get that $\Iso(M,g(t))=\Iso(M,g_0)$ for all $t$.  This fact has recently been proved for the more general class of all complete and with bounded curvature metrics in \cite{Kts}.
\end{remark}

It is easy to check that if
$$
\ip_t=\la P(t)\cdot,\cdot\ra,
$$
where $P(t)$ is the corresponding smooth curve of positive definite operators of $(\pg,\ip)$, then equation (\ref{RFip}) determines the following
ODE for $P$:
\begin{equation}\label{RFh}
\ddt P=-2P\Ricci(\ip_t),
\end{equation}
where $\Ricci(\ip_t):=\Ricci(g(t))(o_{\mu_0}):\pg\longrightarrow\pg$ is the Ricci operator at the origin.

On the other hand, if
$$
\ip_t=\la h(t)\cdot,h(t)\cdot\ra
$$
for some smooth curve $h:(a,b)\longrightarrow\G$ (which there always exists but
it is far from being unique), then $h^th=P$ and so equation (\ref{RFip}) determines the following
ODE for $h$:
\begin{equation}\label{RFh}
h^t\ddt h+\left(\ddt h\right)^th=-2h^th\Ricci(\ip_t).
\end{equation}
Here $h^t$ denotes transpose with respect to our fixed inner product $\ip$.

It follows from the $\Ad(K_{\mu_0})$-invariance of $\ip_t$ that condition (\ref{adkh}) holds for $h(t)$ and $\mu_0$ for all $t$.  Proposition \ref{const} therefore implies that
$$
\mu(t):=\left[\begin{smallmatrix} I&0\\ 0&h(t)\end{smallmatrix}\right]\cdot\mu_0
$$
is a curve in $\hca_{q,n}$, which is equivalent in a way (to be more precisely defined later) to the Ricci flow.  Thus the understanding of its evolution equation and dynamical properties may provide new insights into the study of homogeneous Ricci flows and solitons.  This is the aim of the next section.

\subsection{The bracket flow}\label{lbflow}
The parametrization of homogeneous manifolds as points in the set $\hca_{q,n}$ given in Section \ref{varhs} suggests the following natural question:

\begin{quote}
How does the Ricci flow look on $\hca_{q,n}$?
\end{quote}

More precisely,

\begin{quote}
what is the evolution equation a curve $\mu(t)\in\hca_{q,n}$ must satisfy in order to get that $\left(G_{\mu(t)}/K_{\mu(t)},g_{\mu(t)}\right)$ is a Ricci flow solution up to pullback by time-dependent diffeomorphisms?
\end{quote}

The answer is the content of Theorem \ref{eqfl} below and is given by the following ODE for a curve
$$
\mu(t)\in V_{q+n}=\lamg
$$
of bilinear and skew-symmetric maps, which will be called the {\it bracket flow} from now on:
\begin{equation}\label{BF}
\ddt\mu=-\pi\left(\left[\begin{smallmatrix} 0&0\\ 0&\Ricci_{\mu}
\end{smallmatrix}\right]\right)\mu, \qquad\mu(0)=\mu_0.
\end{equation}
Here $\Ricci_{\mu}$ is defined as in (\ref{ricci}) and $\pi:\glg_{q+n}(\RR)\longrightarrow\End(V_{q+n})$ is the natural representation given by
\begin{equation}\label{actiong}
\pi(A)\mu=A\mu(\cdot,\cdot)-\mu(A\cdot,\cdot)-\mu(\cdot,A\cdot),
\qquad A\in\glg_{q+n}(\RR),\quad\mu\in V_{q+n}.
\end{equation}
We note that $\pi$ is the derivative of the $\Gl_{q+n}(\RR)$-representation defined in (\ref{action}) and $\pi(A)\mu=0$ if and only if $A\in\Der(\mu)$, the Lie algebra of
derivations of the algebra $\mu$.

Equation (\ref{BF}) is well defined as $\Ricci_\mu$ can be computed for any $\mu\in V_{q+n}$ by using formula (\ref{ricci}), and not only for $\mu\in\hca_{q,n}$.    However, as the following lemma shows, this technicality is only needed to define the ODE.  Let $(a,b)$ denote from now on a time interval with $0\in (a,b)$ where a solution $\mu(t)$ exists.

\begin{lemma}\label{muflow}
If $\mu:(a,b)\longrightarrow V_{q+n}$ is a solution to {\rm (\ref{BF})} such that $\mu_0\in\hca_{q,n}$, then $\mu(t)\in\hca_{q,n}$ for all $t\in (a,b)$.
\end{lemma}

\begin{proof}
We must check conditions (h1)-(h4) in Definition \ref{hqn} for $\mu=\mu(t)$.  We have that the right hand side of (\ref{BF}) is tangent to the $\Gl_n(\RR)$-orbit of $\mu$:
$$
\pi\left(\left[\begin{smallmatrix} 0&0\\ 0&\Ricci_{\mu}
\end{smallmatrix}\right]\right)\mu =\dds|_0 e^{s\left[\begin{smallmatrix} 0&0\\ 0&\Ricci_{\mu}
\end{smallmatrix}\right]}\cdot\mu \in T_{\mu}\left(\left[\begin{smallmatrix} I&0\\ 0&\Gl_n(\RR)
\end{smallmatrix}\right]\cdot\mu\right)\subset V_{q+n}, \qquad\forall t\in (a,b).
$$
By a standard ODE theory argument, we get that $\mu(t)\in \left[\begin{smallmatrix} I&0\\ 0&\Gl_n(\RR)
\end{smallmatrix}\right]\cdot\mu_0$ for all
$t$, which implies that condition (h1) holds for $\mu(t)$, and furthermore,
\begin{equation}\label{muk}
\mu(t)|_{\kg\times\kg}=\mu_0, \qquad\forall t\in (a,b).
\end{equation}
Now, for each $Z\in\kg$, set $\psi:=e^{\ad_{\mu_0}{Z}}$.  It is easy to see by using that
$\psi|_{\pg}$ is orthogonal that
$\Ricci_{\psi.\mu}=\psi|_{\pg}\Ricci_{\mu}(\psi|_{\pg})^{-1}$, and thus the curve
$\lambda(t):=\psi\cdot\mu(t)$ satisfies
$$
\ddt\lambda=\psi\cdot\ddt\mu=\psi\cdot\left(-\pi\left(\left[\begin{smallmatrix} 0&0\\
0&\Ricci_{\mu}
\end{smallmatrix}\right]\right)\right)\mu = -\pi\left(\psi\left[\begin{smallmatrix} 0&0\\ 0&\Ricci_{\mu}
\end{smallmatrix}\right]\psi^{-1}\right)\psi\cdot\mu= -\pi\left(\left[\begin{smallmatrix} 0&0\\ 0&\Ricci_{\lambda}
\end{smallmatrix}\right]\right)\lambda.
$$
But $\lambda(0)=\psi\cdot\mu(0)=\mu(0)$ as $\psi\in\Aut(\ggo,\mu_0)$, so that
$\lambda(t)=\mu(t)$ for all $t$ by uniqueness of the solution.  Thus $\psi\in\Aut(\ggo,\mu(t))$ for all $t$, which implies that $\psi|_{\pg}$ commutes with $\Ricci_{\mu}$
and so
\begin{equation}\label{muk2}
[\ad_{\mu_0}{Z}|_{\pg},\Ricci_{\mu}]= 0,\qquad\forall Z\in\kg.
\end{equation}
It follows from (\ref{BF}) that
$$
\ddt\ad_{\mu}{Z}|_{\pg}= -\ad_{\pi\left(\left[\begin{smallmatrix} 0&0\\ 0&\Ricci_{\mu}
\end{smallmatrix}\right]\right)\mu}{Z}|_{\pg}= [\ad_{\mu}{Z}|_{\pg},\Ricci_{\mu}],
$$
and since the same ODE is satisfied by the constant map $\ad_{\mu_0}{Z}|_{\pg}$, it follows that
\begin{equation}\label{admuz}
\ad_{\mu}{Z}|_{\pg}=\ad_{\mu_0}{Z}|_{\pg}, \qquad\forall t\in (a,b), \quad Z\in\kg.
\end{equation}
Conditions (h3) and (h4) are therefore satisfied by $\mu(t)$.  Finally, since $K_{\mu_0}$ is closed in $G_{\mu_0}$, we get that $K_{\mu}$ is also closed in $G_{\mu}$ as it is the image of $K_{\mu_0}$ by the isomorphism between $G_{\mu_0}$ and $G_{\mu}$ with derivative at the identity of the form $\left[\begin{smallmatrix} I&0\\ 0&h(t)
\end{smallmatrix}\right]$.  This implies that condition (h2) holds, concluding the proof of the lemma.
\end{proof}

We conclude from Lemma \ref{muflow} that a homogeneous space $(G_{\mu(t)}/K_{\mu(t)},g_{\mu(t)})$ can indeed
be associated to each $\mu(t)$ in a bracket flow solution provided that $\mu_0\in\hca_{q,n}$.  We now show that the Ricci flow and the bracket flow are intimately related.

For a given simply connected homogeneous manifold $(M,g_0)=\left(G_{\mu_0}/K_{\mu_0},g_{\mu_0}\right)$, $\mu_0\in\hca_{q,n}$, let us consider the following one-parameter families:
\begin{equation}\label{3rm}
(M,g(t)), \qquad \left(G_{\mu_0}/K_{\mu_0},g_{\ip_t}\right), \qquad \left(G_{\mu(t)}/K_{\mu(t)},g_{\mu(t)}\right),
\end{equation}
where $g(t)$, $\ip_t$ and $\mu(t)$ are the solutions to the Ricci flows (\ref{RF}), (\ref{RFip}) and the bracket flow (\ref{BF}), respectively.  Recall that $\ggo=\kg\oplus\pg$ is a reductive decomposition for any of the homogeneous spaces involved.

\begin{theorem}\label{eqfl}
There exist time-dependent diffeomorphisms $\vp(t):M\longrightarrow G_{\mu(t)}/K_{\mu(t)}$ such that
$$
g(t)=\vp(t)^*g_{\mu(t)}, \qquad\forall t\in (a,b).
$$
Moreover, if we identify $M=G_{\mu_0}/K_{\mu_0}$, then $\vp(t):G_{\mu_0}/K_{\mu_0}\longrightarrow G_{\mu(t)}/K_{\mu(t)}$ can be chosen as the equivariant diffeomorphism determined by the Lie group isomorphism between $G_{\mu_0}$ and $G_{\mu(t)}$ with derivative $\tilde{h}:=\left[\begin{smallmatrix} I&0\\ 0&h \end{smallmatrix}\right]:\ggo\longrightarrow\ggo$, where $h(t):=d\vp(t)|_{o_{\mu_0}}:\pg\longrightarrow\pg$ is the solution to any of the following systems of ODE's:

\begin{itemize}
\item[(i)] $\ddt h=-h\Ricci(\ip_t)$, $\quad h(0)=I$.

\item[(ii)] $\ddt h=-\Ricci_{\mu(t)}h$, $\quad h(0)=I$.
\end{itemize}
The following conditions also hold:
\begin{itemize}
\item[(iii)] $\ip_t=\la h\cdot,h\cdot\ra$.

\item[(iv)] $\mu(t)=\tilde{h}\mu_0(\tilde{h}^{-1}\cdot,\tilde{h}^{-1}\cdot)$.
\end{itemize}
\end{theorem}

\begin{remark}\label{remeq}
Before proceeding with the proof, it is worth pointing out the following useful facts which are direct consequences of the theorem:

\begin{itemize}
\item The Ricci flow $g(t)$ and the bracket flow $g_{\mu(t)}$ differ only by pullback by time-dependent diffeomorphisms.

\item They are equivalent in the following sense: each one can be obtained from the other by solving the corresponding ODE in part (i) or (ii) and applying parts (iv) or (iii), accordingly.

\item The maximal interval of time where a solution exists is therefore the same for both flows.

\item At each time $t$, the Riemannian manifolds in (\ref{3rm}) are all isometric to each other, so that the behavior
of the curvature and of any other Riemannian invariant along the
Ricci flow $g(t)$ can be studied along the bracket flow $g_{\mu(t)}$.

\item If some sequence $\mu(t_k)$ (or a suitable normalization) converges to $\lambda\in\hca_{q,n}$, then we can apply the results described in Section \ref{convergence} to get convergence or subconvergence of the metrics $g_{\mu_k}\to g_{\lambda}$ relative to natural notions of convergence, as infinitesimal, local or pointed.
\end{itemize}
\end{remark}

\begin{proof}
We will first prove that part (i) implies all the other statements in the theorem.  Consider the solution $h=h(t)\in\G$ to the ODE
$$
\ddt h=-h\Ricci(\ip_t),\qquad h(0)=I,
$$
where $\Ricci(\ip_t):=\Ricci(g(t))(0)$, which is defined on the same interval of time as $g(t)$ by a standard result in ODE theory ($h(t)$ is easily seen to be invertible for all $t$).  If $\ipp_t:=\la h(t)\cdot,h(t)\cdot\ra$ and $h':=\ddt h(t)$ then
\begin{align}
\ddt\ipp_t &= \la h'\cdot,h\cdot\ra+\la h\cdot,h'\cdot\ra \notag\\
&= -\la h\Ricci(\ip_t)\cdot,h\cdot\ra -\la h\cdot,h\Ricci(\ip_t)\cdot\ra \label{ddtip}\\
&= -(\Ricci(\ip_t)\cdot,\cdot)_t -(\cdot,\Ricci(\ip_t)\cdot)_t. \notag
\end{align}
On the other hand, since $\Ricci(\ip_t)$ is symmetric with respect to $\ip_t$, it follows from (\ref{RFip}) that $\ip_t$ satisfies
$$
\ddt\ip_t=-2\ricci(\ip_t)=-2\la\Ricci(\ip_t)\cdot,\cdot\ra_t= -\la\Ricci(\ip_t)\cdot,\cdot\ra_t-\la\cdot,\Ricci(\ip_t)\cdot\ra_t.
$$
Thus $\ipp_t$ and $\ip_t$, as curves in the manifold $\G/\Or(n)$ of inner products on $\pg$, satisfy the same ODE and $\ipp_0=\ip_0=\ip$, and so part (iii) follows by uniqueness of the solution.  This also implies that condition (\ref{adkh}) holds for $h(t)$ and $\mu_0$, from we get that
$$
\vp(t):(G_{\mu_0}/K_{\mu_0},g_{\ip_t})\longrightarrow \left(G_{\lambda(t)}/K_{\lambda(t)},g_{\lambda(t)}\right)
$$
is an isometry for the curve $\lambda(t):=\tilde{h}(t).\mu_0$ for all $t$ (see Proposition \ref{const}).  In particular,  $\Ricci_{\lambda(t)}=h(t)\Ricci(\ip_t)h(t)^{-1}$, or equivalently, $h'=-\Ricci_{\lambda(t)}h$, and thus
\begin{align}
\ddt\lambda &= \tilde{h}'\mu_0(\tilde{h}^{-1}\cdot,\tilde{h}^{-1}\cdot) -\tilde{h}\mu_0(\tilde{h}^{-1}\tilde{h}'\tilde{h}^{-1}\cdot,\tilde{h}^{-1}\cdot) -\tilde{h}\mu_0(\tilde{h}^{-1}\cdot,\tilde{h}^{-1}\tilde{h}'\tilde{h}^{-1}\cdot) \notag\\
&= (\tilde{h}'\tilde{h}^{-1})\tilde{h}\mu_0(\tilde{h}^{-1}\cdot,\tilde{h}^{-1}\cdot)-\tilde{h}\mu_0(\tilde{h}^{-1}(\tilde{h}'\tilde{h}^{-1})\cdot,\tilde{h}^{-1}\cdot) -\tilde{h}\mu_0(\tilde{h}^{-1}\cdot,\tilde{h}^{-1}(\tilde{h}'\tilde{h}^{-1})\cdot) \label{ddtmu}\\
&= \pi(\tilde{h}'\tilde{h}^{-1})\lambda= -\pi\left(\left[\begin{smallmatrix} 0&0\\ 0&\Ricci_{\lambda} \notag
\end{smallmatrix}\right]\right)\lambda.
\end{align}
The curve $\lambda(t)$ is therefore a solution to the bracket flow and since $\lambda(0)=\mu_0$, we obtain that $\mu(t)=\lambda(t)$ for all $t$, from which parts (ii) and (iv) follow.  We also deduce that
$$
\vp(t):(M,g(t))=(G_{\mu_0}/K_{\mu_0},g_{\ip_t})\longrightarrow (G_{\mu(t)}/K_{\mu(t)},g_{\mu(t)})
$$
is an isometry.

Let us now assume that part (ii) holds, and so $h(t)$ is defined on the same time interval as $g(t)$.  By carrying out the same computation as in (\ref{ddtmu}) for $\mu$, we get that $\tilde{h}(t)\cdot\mu_0$ is a solution to the bracket flow starting at $\mu_0$.  Thus $\tilde{h}(t)\cdot\mu_0=\mu(t)\in\hca_{q,n}$ for all $t$ (i.e. part (iv) holds), from which easily follows that $h(t)$ satisfies (\ref{adkh}) and therefore $\ipp_t:=\la h\cdot,h\cdot\ra$ defines a $G_{\mu_0}$-invariant metric on $G_{\mu_0}/K_{\mu_0}$ for all $t$.  Moreover, we have that
$$
\vp(t):(G_{\mu_0}/K_{\mu_0},g_{\ipp_t})\longrightarrow \left(G_{\mu(t)}/K_{\mu(t)},g_{\mu(t)}\right)
$$
is an isometry for all $t$ (see Proposition \ref{const}), and so  $\Ricci_{\mu(t)}=h(t)\Ricci(\ipp_t)h(t)^{-1}$, or equivalently, $h'=-h\Ricci(\ipp_t)$.  This implies that $\ipp_t$ is a solution to the Ricci flow (\ref{RFip}) by arguing as in (\ref{ddtip}), and consequently, $\ipp_t=\ip_t$ for all $t$.  In this way, parts (i) and (iii) follow, concluding the proof of the theorem.
\end{proof}

Since $\mu(t)|_{\kg\times\ggo}=\mu_0|_{\kg\times\ggo}$ for all $t\in (a,b)$ (see (\ref{muk})
and (\ref{admuz})), we have that only $\mu(t)|_{\pg\times\pg}$ is actually evolving, and so the bracket flow equation (\ref{BF}) can be rewritten as the simpler system
\begin{equation}\label{BFsis}
\left\{\begin{array}{ll}
\ddt\mu_{\kg}=\mu_{\kg}(\Ricci_{\mu}\cdot,\cdot)+\mu_{\kg}(\cdot,\Ricci_{\mu}\cdot), & \\
& \mu_{\kg}(0)+\mu_{\pg}(0)=\mu_0|_{\pg\times\pg},\\
\ddt\mu_{\pg}=-\pi_n(\Ricci_{\mu})\mu_{\pg}, &
\end{array}\right.
\end{equation}
where $\mu_{\kg}$ and $\mu_{\pg}$ are the components of $\mu|_{\pg\times\pg}$ as in (\ref{decmu}) and $\pi_n:\glg_n(\RR)\longrightarrow\End(V_n)$ is the  representation defined in (\ref{actiong}) for $q=0$.

\begin{example}\label{ex0-3BF}
Let $\mu=\mu_{a,b,c}$ be the Lie bracket in $\hca_{0,3}=\lca_3$ defined by
$$
\mu(X_2,X_3)=aX_1, \qquad \mu(X_3,X_1)=bX_2, \qquad \mu(X_1,X_2)=cX_3.
$$
If follows from the formula for the Ricci operator $\Ricci_\mu$ given in Example \ref{ex0-3ricci} that this family is invariant under the bracket flow, which is equivalent to the following ODE system for the variables $a(t),b(t),c(t)$:
$$
\left\{\begin{array}{l}
\ddt a= \left(-\unm(3a^2-b^2-c^2)+ab+ac-bc\right)a, \\ \\
\ddt b= \left(-\unm(3b^2-a^2-c^2)+ab-ac+bc\right)b,  \\ \\
\ddt c= \left(-\unm(3c^2-a^2-b^2)-ab+ac+bc\right)c.
\end{array}\right.
$$
We refer to \cite{GlcPyn} for a complete qualitative study of this dynamical system, including some nice phase plane pictures.  One can use Theorem \ref{convmu3} to get  subsequences which are convergent in the pointed sense from all convergent (normalized) bracket flow solutions obtained in \cite{GlcPyn}, and even smooth convergence in many cases.
\end{example}

\begin{example}\label{ex1-3BF}
The Ricci operator of the bracket $\mu=\mu_{a,b,c}\in\hca_{1,3}$ given by
$$
\left\{\begin{array}{lll}
\mu(X_3,Z_1)=X_2, & \mu(X_2,X_3)=aX_1+bZ_1, & \mu(X_3,X_1)=cX_2, \\
\mu(Z,X_2)=X_3, &                           & \mu(X_1,X_2)=cX_3,
\end{array}\right.
$$
was computed in Example \ref{ex1-3ricci}.  It is easy to see that the bracket flow leaves this family invariant and is equivalent to the ODE system
$$
\left\{\begin{array}{l}
\ddt a=(-\tfrac{3}{2}a^2+2b+2ac)a, \\ \\
\ddt b=(-a^2+2b+2ac)b, \\ \\
\ddt c=\unm a^2c.
\end{array}\right.
$$
Notice that $\mu_{\kg}$ is given by $\mu_{\kg}(X_2,X_3)=bZ_1$, and $a,c$ are the structural constants of $\mu_{\pg}$.  It follows at once from the equations that all the coordinate axes and planes are invariant by the flow, and it is also easy to see that if $b\ne 0$, then $\frac{ac}{b}\equiv constant$, which gives rise to more invariant subsets.  The case $c=0$ will be studied in detail in Section \ref{exdim3}, where some interesting convergence features appear (see \cite{Arr} for a qualitative analysis of all possible cases).
\end{example}

\subsection{Some evolution equations along the bracket flow}
We study in this section how the different parts of the Ricci curvature of $g_{\mu(t)}$ evolve along the bracket flow (see Section \ref{riccisec} for most of the notation used in what follows).

\begin{lemma}\label{prop}
For any $\mu\in\hca_{q,n}$, the following conditions hold:
\begin{itemize}
\item[(i)] If $\delta_{\mu_{\pg}}:\glg_n(\RR)\longrightarrow\lamp$ is defined by $\delta_{\mu_{\pg}}(E)=-\pi(E)\mu_{\pg}$, then
$$
\delta_{\mu_{\pg}}(I)=\mu_{\pg}, \qquad \delta_{\mu_{\pg}}^t(\mu_{\pg})=-4\mm_{\mu_{\pg}},
$$
where $\delta_{\mu_{\pg}}^t:\lamp\longrightarrow\glg_n(\RR)$ is the transpose of $\delta_{\mu_{\pg}}$.
\item[ ]
\item[(ii)] $\tr{\mm_{\mu_{\pg}}}=-\unc\|\mu_{\pg}\|^2$.
\item[ ]

\item[(iii)] $\tr{\mm_{\mu_{\pg}}D}=\tr{B_{\mu}D}=0$ for any $\left[\begin{smallmatrix} 0&0\\ 0&D
\end{smallmatrix}\right]\in\Der(\mu)$.
\item[ ]

\item[(iv)] $B_{\tilde{h}\cdot\mu}=(h^{-1})^tB_{\mu}h^{-1}$ for any $\tilde{h}=\left[\begin{smallmatrix} I&0\\ 0&h
\end{smallmatrix}\right]\in\Gl_{q+n}(\RR)$, $h\in\Gl_n(\RR)$.
\item[ ]

\item[(v)] $H_{h\cdot\mu_{\pg}}=(h^{-1})^t(H_{\mu_{\pg}})$ for any $h\in\Gl_n(\RR)$.
\item[ ]

\item[(vi)] $\mu(Z,H_{\mu_{\pg}})=0$ for any $Z\in\kg$.
\item[ ]

\item[(vii)] $[\ad_{\mu}{Z}|_{\pg},\Ricci_\mu]=[\ad_{\mu}{Z}|_{\pg},\mm_{\mu_{\pg}}]=[\ad_{\mu}{Z}|_{\pg}, B_{\mu}]=[\ad_{\mu}{Z}|_{\pg},S(\ad_{\mu_{\pg}}{H_{\mu_{\pg}}})]=0$, for any $Z\in\kg$.
\end{itemize}
\end{lemma}

\begin{proof}
The proofs of parts (i)-(v) easily follow by only using the definitions and (\ref{Rmm}), and the vanishing of all the brackets in (vii) can be proved by using that $e^{t\ad_{\mu}{Z}}$ is an automorphism of $(\ggo,\mu)$ which is orthogonal on $\pg$ for all $t$ and so it is easily seen to commute with all the operators on the right in each bracket.

Finally, part (vi) follows from the fact that
$$
\la\mu(Z,H_{\mu_{\pg}}),X\ra=-\la H_{\mu_{\pg}},[Z,X]\ra=-\tr{[\ad_{\mu}{Z},\ad_{\mu}{X}]}=0, \qquad\forall X\in\pg,
$$
concluding the proof of the lemma.
\end{proof}

Let us also denote by $\mm$, $B$ and $H$ the maps defined by
$$
\begin{array}{c}
\mm:\lamp\longrightarrow\glg_n(\RR), \qquad B:\lamg\longrightarrow\glg_n(\RR)\qquad H:\lamp\longrightarrow\pg, \\ \\
\mm(\lambda):=\mm_{\lambda}, \qquad B(\lambda):=B_{\lambda}, \qquad H(\lambda):=H_{\lambda}.
\end{array}
$$
If $E\in\glg_n(\RR)$ satisfies $E^t=E$, then it
follows from (\ref{Rmm}) and (\ref{actiong}) that
\begin{align*}
\la d \mm|_{\lambda}\delta_{\lambda}(A),E\ra &= \la\ddt|_0
\mm_{e^{-tA}.\lambda},E\ra
= \unc\ddt|_0\la\pi_n(E)e^{-tA}.\lambda,e^{-tA}.\lambda\ra \\
&= \unm\la\pi_n(E)\lambda,\delta_{\lambda}(A)\ra =
-\unm\la\delta_{\lambda}^t\delta_{\lambda}(A),E\ra.
\end{align*}
This implies that
\begin{equation}\label{dR}
d \mm|_{\lambda}\delta_{\lambda}(A)=-\unm\Delta_{\lambda}(A), \qquad\forall
A\in\glg_n(\RR), \quad \lambda\in\lamp,
\end{equation}
where
$$
\Delta_{\lambda}:=S\circ\delta_{\lambda}^t\delta_{\lambda}:\glg_n(\RR)\longrightarrow\glg_n(\RR),
$$
and $S$ is as in $(\ref{sym})$.  For the function $B$ we use Lemma \ref{prop}, (iv) to easily obtain that
\begin{equation}\label{dB}
d B|_{\lambda}\delta_{\lambda}(\tilde{A})=A^tB_{\lambda}+B_{\lambda}A,  \qquad\forall
A\in\glg_n(\RR), \quad \lambda\in\lamg,
\end{equation}
where we set $\tilde{A}=\left[\begin{smallmatrix} 0&0\\ 0&A
\end{smallmatrix}\right]\in\glg_{q+n}(\RR)$ for any $A\in\glg_n(\RR)$, that is, $\tilde{A}|_{\kg}=0$, $\tilde{A}|_{\pg}=A$.  Concerning $H$, it follows from Lemma \ref{prop}, (v) that
\begin{equation}\label{dH}
d H|_{\lambda}\delta_{\lambda}(A)=A^t(H_{\lambda}),  \qquad\forall A\in\glg_n(\RR), \quad \lambda\in\lamp.
\end{equation}

Let $\mu(t)\in\hca_{q,n}$ be a solution to the bracket flow (\ref{BF}).  We introduce the following notation in
order to simplify the formulas for the ODE's we need to study:
\begin{equation}\label{not}
\begin{array}{lllll}
\Ricci(t):=\Ricci_{\mu(t)}, && \mm(t):=M_{\mu_{\pg}(t)}, && B(t):=B_{\mu(t)}, \\ \\
H(t):=H_{\mu_{\pg}(t)}, && U(t):=S(\ad_{\mu_{\pg}(t)}{H(t)}), &&
\scalar(t):=\tr{\Ricci(t)}, \\ \\
\Delta(t):=\Delta_{\mu_{\pg}(t)}. &&
\end{array}
\end{equation}

\begin{proposition}\label{eqs}
The bracket flow equation {\rm (\ref{BF})} for $\mu(t)$ produces the following evolution equations:
\begin{itemize}
\item[(i)] $\ddt\Ricci=-\unm\Delta(\Ricci)-\unm(B\Ricci+\Ricci B)- 2S(\ad_{\mu_{\pg}}{\Ricci(H)}) - S([\ad_{\mu_{\pg}}{H},\Ricci])$.
\item[ ]
\item[(ii)] $\ddt\mm=-\unm\Delta(\Ricci)$.
\item[ ]
\item[(iii)] $\ddt B=B\Ricci+\Ricci B$.
\item[ ]
\item[(iv)] $\ddt H=\Ricci(H)$.
\item[ ]
\item[(v)] $\ddt U= 2S(\ad_{\mu_{\pg}}{\Ricci(H)}) + S([\ad_{\mu_{\pg}}{H},\Ricci])$.
\item[ ]
\item[(vi)] $\ddt\scalar=2\tr{\Ricci^2}=2\|\Ricci\|^2$.
\item[ ]
\item[(vii)] $\ddt\|\mu_{\pg}\|^2=-8\tr{\Ricci\mm}$.
\item[ ]
\item[(viii)] $\ddt\tr{B}=2\tr{\Ricci B}$.
\item[ ]
\item[(ix)] $\ddt\| H\|^2=-2\tr{S(\ad_{\mu_{\pg}}{H})^2}$.
\end{itemize}

\end{proposition}

\begin{proof}
We use (\ref{dR}) in the last equality of the following line to prove part (ii),
$$
\ddt\mm= d \mm|_{\mu_{\pg}}\ddt\mu_{\pg} = d \mm|_{\mu_{\pg}}\delta_{\mu_{\pg}}(\Ricci) = -\unm\Delta(\Ricci),
$$
and parts (iii) and (iv) follow similarly by using (\ref{dB}) and (\ref{dH}), respectively.

It follows from part (iv) and the fact that $S$ and $\ad$ are linear that
\begin{align*}
\ddt U &= S\left(\ddt \ad_{\mu_{\pg}}{H}\right) = S\left(\ad_{\delta_{\mu_{\pg}}}(\Ricci){H} + \ad_{\mu_{\pg}}{\Ricci(H)}\right) \\
&= S\left(\ad_{\mu_{\pg}}{\Ricci(H)} + \ad_{\mu_{\pg}}{H}\circ\Ricci - \Ricci\circ\ad_{\mu_{\pg}}{H} + \ad_{\mu_{\pg}}{\Ricci(H)}\right) \\
&= 2S(\ad_{\mu_{\pg}}{\Ricci(H)}) + S([\ad_{\mu_{\pg}}{H},\Ricci]),
\end{align*}
which proves part (v).  It is clear that part (i) follows from part (ii), (iii) and (v).

We have that
\begin{align*}
\la\Ricci(H),H\ra &= \la\mm(H),H\ra-\unm\la B(H),H\ra  \\
&= -\unm\tr{(\ad_{\mu_{\pg}}{H})^t\ad_{\mu_{\pg}}{H}} + \unc\sum\tr{\ad_{\mu}{\mu(X_i,X_j)}} -\unm\tr{(\ad_{\mu_{\pg}}{H})^2} \\
&= -\tr{S(\ad_{\mu_{\pg}}{H})\ad_{\mu_{\pg}}{H}} = -\tr{S(\ad_{\mu_{\pg}}{H})^2},
\end{align*}
and thus by Lemma \ref{prop}, (iii), we get
\begin{equation}\label{RicH}
\la\Ricci(H),H\ra = -\tr{S(\ad_{\mu_{\pg}}{H})^2} = \la\Ricci,S(\ad_{\mu_{\pg}}{H})\ra.
\end{equation}

We now use part (i) and (\ref{RicH}) to prove (iv) as follows:
\begin{align*}
\ddt R =& \tr{\ddt \Ricci} = -\unm\tr{\delta_{\mu_{\pg}}^t\delta_{\mu_{\pg}}(\Ricci)} -2\tr{B\Ricci} -2\tr{\ad_{\mu_{\pg}}{\Ricci(H)}} \\
=& -\unm\la\Ricci,\delta_{\mu_{\pg}}^t(\mu_{\pg})\ra -2\la\Ricci,B\ra -2\la\Ricci(H),H\ra \\
=&  2\la\Ricci,\mm\ra-2\la\Ricci,B\ra-2\la\Ricci,S(\ad_{\mu_{\pg}}{H})\ra =2\tr{\Ricci^2}.
\end{align*}

By using (ii) we obtain part (vii):
$$
\ddt\|\mu_{\pg}\|^2= 2\la\mu_{\pg},\ddt\mu_{\pg}\ra =2\la\mu_{\pg},\delta_{\mu_{\pg}}(\Ricci)\ra = -8\tr{\Ricci\mm}.
$$
Finally, parts (viii) and (x) follow from (iii) and (iv), (\ref{RicH}), respectively, concluding the proof of the proposition.
\end{proof}

A few comments are in order here:

\begin{itemize}
\item In the general case, along a Ricci flow solution $g(t)$, the scalar curvature $R=R(g(t))$ evolves according to
$$
\dpar R=\Delta(R)+2\|\ricci\|^2,
$$
where $\Delta$ is the Laplacian of $(M,g(t))$ (see e.g. \cite[Lemma 6.7]{ChwKnp}).  Since at each fixed time $R$ is constant as a function on $M$ for homogeneous manifolds, one obtains the evolution given in Proposition \ref{eqs}, (vi).  In particular, the scalar curvature is always strictly increasing in the homogeneous case, unless $g(t)\equiv g_0$ is flat (recall that a homogeneous manifold is flat if and only if it is Ricci flat; see \cite{AlkKml}).

\item If $\mu(t)\to\lambda\in\hca_{q,n}$, as $t\to\infty$, then $R(t)\to R(\lambda)$ (increasing) and so $\Ricci_\lambda=0$, that is, $(G_\lambda/K_\lambda,g_\lambda)$ is flat.

\item It follows from Proposition \ref{eqs}, (ix) that homogeneous spaces become `more unimodular' (i.e. $H=0$) while evolve by the Ricci flow.
\end{itemize}

\subsection{Normalized flows}\label{norm}
Let $(M,g_0)$ be a Riemannian manifold.  By rescaling the metric and reparametrizing the time
variable $t$, one can transform the Ricci flow (\ref{RF}) into an {\it $r$-normalized Ricci flow}
\begin{equation}\label{RFrn}
\dpar g^r(t)=-2\ricci(g^r(t))-2r(t)g^r(t),\qquad g^r(0)=g_0,
\end{equation}
for some {\it normalization function} $r(t)$ which may depend on $g^r(t)$.  It is easy to see that a solution has the form $g^r(t)=a(t)g(b(t))$ with $g(s)$ the Ricci flow and thus a scalar Riemannian invariant may remain constant or bounded in time as a result of an appropriate choice of the function $r(t)$.  Normalizations are therefore very useful to prevent a Ricci flow solution from a finite-time singularity as well as from converging to a flat metric.  Usually, the challenge is to be able to prevent both by using the same normalization.

In the case $(M,g)$ is homogeneous, say $(M,g)=\left(G_{\mu_0}/K_{\mu_0},g_{\mu_0}\right)$, $\mu_0\in\hca_{q,n}$, we have that the flow (\ref{RFrn}) is equivalent to
\begin{equation}\label{RFiprn}
\ddt\ip^r_t=-2\ricci(\ip^r_t)-2r(t)\ip^r_t, \qquad \ip^r_0=\ip.
\end{equation}
This motivates the definition of the {\it $r$-normalized bracket flow} for $\mu^r=\mu^r(t)$ and $r=r(t)$ by
\begin{equation}\label{BFrn}
\ddt\mu^r=-\pi\left(\left[\begin{smallmatrix} 0&0\\ 0&\Ricci_{\mu^r}+rI
\end{smallmatrix}\right]\right)\mu^r, \qquad \mu^r(0)=\mu_0,
\end{equation}
or equivalently, as the system analogous to (\ref{BFsis}) given by
\begin{equation}\label{BFrnsis}
\left\{\begin{array}{ll}
\ddt\mu^r_{\kg}=\mu^r_{\kg}(\Ricci_{\mu^r}\cdot,\cdot)+\mu^r_{\kg}(\cdot,\Ricci_{\mu^r}\cdot) +2r\mu_{\kg}^r(\cdot,\cdot), & \\
& \mu^r_{\kg}(0)+\mu^r_{\pg}(0)=\mu_0|_{\pg\times\pg},\\
\ddt\mu^r_{\pg}=-\pi_n(\Ricci_{\mu^r}+rI)\mu^r_{\pg} =-\pi_n(\Ricci_{\mu^r})\mu^r_{\pg} +r\mu_{\pg}^r. &
\end{array}\right.
\end{equation}

Given a continuous (or just integrable) normalization function $r$, we consider the solutions to the ODE's $\{ c'=rc, \; c(0)=1\}$ and $\{ \tau'=c^2,\; \tau(0)=0\}$, that is,
\begin{equation}\label{defctau}
r\quad\rightsquigarrow \quad c(t):=e^{\int_0^t r(s)\; ds}, \quad \tau(t)=\int_0^t c^2(s)\; ds.
\end{equation}
Notice that $c(t)>0$ and $\tau'(t)>0$ for all $t$.  If $g(t)$, $\ip_t$ and $\mu(t)$ denote the (unnormalized) Ricci and bracket flows (\ref{RF}), (\ref{RFip}) and (\ref{BF}), respectively, then it is straightforward to check that
\begin{equation}\label{rnsol}
\begin{array}{c}
g^r(t)=\frac{1}{c^2(t)}g(\tau(t)), \quad \ip^r_t=\frac{1}{c^2(t)}\ip_{\tau(t)}, \\ \\ \mu^r(t)=c(t)\cdot\mu(\tau(t)), \qquad\mbox{i.e.}\quad \mu^r_{\kg}(t)=c^2(t)\mu_{\kg}(\tau(t)), \quad \mu^r_{\pg}(t)=c(t)\mu_{\pg}(\tau(t)).
\end{array}
\end{equation}

Let $[0,T)$ be the maximal interval of forward time for the bracket flow solution $\mu(t)$, with $T\in\RR_{>0}\cup\{\infty\}$.

\begin{lemma}\label{maxtime}
If the $r$-normalized bracket flow solution $\mu^r(t)$ is defined for $t\in[0,\infty)$ and $\tau(t)\to T_1<T$, as $t\to\infty$, then $\mu^r(t)\to 0$, as $t\to\infty$.
\end{lemma}

\begin{proof}
Under these hypothesis, one has that $\mu(\tau(t))\to\mu(T_1)$ and $c^2(t)=\tau'(t)\to 0$, as $t\to\infty$, and thus $\mu^r(t)\to 0$ by (\ref{rnsol}).
\end{proof}

In particular, if $\mu^r(t)\to\lambda\ne 0$, as $t\to\infty$, then $\tau(t)\to T$ and $\lambda$ is therefore providing information on the behavior of $\mu(t)$ close to the singularity $T$.

For the corresponding one-parameter families of homogeneous manifolds:
\begin{equation}\label{3rmrn}
(M,g^r(t)), \qquad \left(G_{\mu_0}/K_{\mu_0},g_{\ip^r_t}\right), \qquad \left(G_{\mu^r(t)}/K_{\mu^r(t)},g_{\mu^r(t)}\right),
\end{equation}
one can prove the following result in much the same way as Theorem \ref{eqfl}, or alternatively, by just defining $h^r$ as in part (v) below, which make all the statements easy to check.

\begin{theorem}\label{eqflrn}
There exist time-dependent diffeomorphisms $\vp^r(t):M\longrightarrow G_{\mu^r(t)}/K_{\mu^r(t)}$ such that
$$
g^r(t)=\vp^r(t)^*g_{\mu^r(t)}, \qquad\forall t\in (a,b).
$$
Moreover, if we identify $M=G_{\mu_0}/K_{\mu_0}$, then $\vp^r(t):G_{\mu_0}/K_{\mu_0}\longrightarrow G_{\mu^r(t)}/K_{\mu^r(t)}$ can be chosen as the equivariant diffeomorphism determined by a Lie group isomorphism between $G_{\mu_0}$ and $G_{\mu(t)}$ with derivative $\tilde{h^r}:=\left[\begin{smallmatrix} I&0\\ 0&h^r \end{smallmatrix}\right]:\ggo\longrightarrow\ggo$, where $h^r(t):=d\vp(t)|_{o_{\mu_0}}:\pg\longrightarrow\pg$ is the solution to any of the following ODE systems:

\begin{itemize}
\item[(i)] $\ddt h^r=-h^r(\Ricci(\ip^r_t)+r(t)I)$, $\quad h^r(0)=I$.

\item[(ii)] $\ddt h^r=-(\Ricci_{\mu^r(t)}+r(t)I)h^r$, $\quad h^r(0)=I$.
\end{itemize}
The following conditions also hold:
\begin{itemize}
\item[(iii)] $\ip^r_t=\la h^r\cdot,h^r\cdot\ra$.

\item[(iv)] $\mu^r(t)=\tilde{h^r}\mu_0(\tilde{h^r}^{-1}\cdot,\tilde{h^r}^{-1}\cdot)$.

\item[(v)] $h^r(t):= \frac{1}{c(t)} h(\tau(t))$, where $h(t)$ is defined as in {\rm Theorem \ref{eqfl}} and $c$ and $\tau$ are given by {\rm (\ref{defctau})}.
\end{itemize}
\end{theorem}

It is worth mentioning at this point that by Theorem \ref{convmu}, any convergence $\mu^r(t)\to\lambda\in\hca_{q,n}$, as $t\to\infty$, we may get from some normalized flow of the form (\ref{BFrn}) gives rise to infinitesimal convergence $\left(G_{\mu^r(t)}/K_{\mu^r(t)},g_{\mu^r(t)}\right)\to\left(G_{\lambda}/K_{\lambda},g_{\lambda}\right)$, as well as to local convergence (and hence pointed subconvergence) provided the Lie injectivity radius remains uniformly bounded from below, i.e. $\inf\limits_{t\in[0,\infty)} r_{\mu(t)}>0$.  Recall also that $g^r(t)$ and $g_{\mu^r(t)}$ are isometric for all $t$, and so any geometric quantity constant or bounded in time for the normalized bracket flow $\mu^r(t)$ will also be so for the normalized Ricci flow $g^r(t)$.

We now give the evolutions of the quantities associated to the Ricci curvature along an $r$-normalized bracket flow.

\begin{proposition}\label{eqsrn}
If we replace $\mu$ by $\mu^r$ in every formula given in {\rm (\ref{not})}, then the $r$-normalized bracket flow equation {\rm (\ref{BFrn})} for $\mu^r(t)$ produces the following evolution equations:
\begin{itemize}
\item[(i)] \begin{align*} \ddt\Ricci=&-\unm\Delta(\Ricci)-\unm(B\Ricci+\Ricci B) \\ &- 2S(\ad_{\mu_{\pg}^r}{\Ricci(H)}) - S([\ad_{\mu_{\pg}^r}{H},\Ricci]) + 2r\Ricci.\end{align*}
\item[ ]
\item[(ii)] $\ddt\mm=-\unm\Delta(\Ricci) + 2r\mm$.
\item[ ]
\item[(iii)] $\ddt B=B\Ricci+\Ricci B + 2rB$.
\item[ ]
\item[(iv)] $\ddt H=\Ricci(H) + rH$.
\item[ ]
\item[(v)] $\ddt U= 2S(\ad_{\mu_{\pg}^r}{\Ricci(H)}) + S([\ad_{\mu_{\pg}^r}{H},\Ricci]) + 2rS(\ad_{\mu_{\pg}^r}{H})$.
\item[ ]
\item[(vi)] $\ddt\scalar=2\tr{\Ricci^2} + 2rR = 2\|\Ricci\|^2 + 2rR$.
\item[ ]
\item[(vii)] $\ddt\|\mu_{\pg}^r\|^2=-8\tr{\Ricci\mm} + 2r\|\mu_{\pg}^r\|^2$.
\item[ ]
\item[(viii)] $\ddt\tr{B}=2\tr{\Ricci B} + 2r\tr{B}$.
\item[ ]
\item[(ix)] $\ddt\| H\|^2=-2\tr{S(\ad_{\mu_{\pg}^r}{H})^2} + 2r\| H\|^2$.
\end{itemize}
\end{proposition}

\begin{proof}
We can replace $\Ricci$ by $\Ricci+rI$ everywhere in the proof of Proposition \ref{eqs}, and obtain a proof for each item in exactly the same way.  One only need to be careful in the proof of part (vi), where only the $\Ricci$ on the left must be changed.
\end{proof}

\begin{example}\label{norm-vol}
In order to get the volume element of $g_{\mu^r}$ constant in time (i.e. $\det{h^r}\equiv 1$) one can take $r(t)=-\frac{1}{n}\scalar(\mu^r)$.  Indeed, by using Theorem \ref{eqflrn}, (ii), we obtain
$$
\ddt\det{h^r} = \det{h^r}\tr{\left((h^r)^{-1}\ddt h^r\right)} = -\det{h^r}\tr{(\Ricci_{\mu^r}+rI)}=0,
$$
and so $\det{h^r}\equiv 1$ as $h^r(0)=I$.  A useful property of this normalization is that the scalar curvature is still increasing. Indeed, by Theorem \ref{eqsrn}, (vi), we have that
$$
\ddt\scalar(\mu^r)= 2\left(\tr{\Ricci_{\mu^r}^2} - \frac{1}{n}R(\mu_r)^2\right) \geq 0,
$$
and equality holds if and only if $\mu_0$ is Einstein, in which case $\mu^r(t)\equiv\mu_0$, $t\in(-\infty,\infty)$.
\end{example}

\begin{example}\label{norm-R}
In the homogeneous case, the scalar curvature is just a single number attached to the metric, providing a natural curvature quantity to have fixed along the flow.  It follows from Theorem \ref{eqsrn}, (vi) that by normalizing with $r(t)=-\frac{\tr{\Ricci_{\mu^r}^2}}{\scalar(\mu^r)}$, we get $\scalar(\mu^r)\equiv\scalar(\mu_0)$  provided $\scalar(\mu_0)\ne 0$.
\end{example}

\begin{example}\label{norm-norma}
If the normalization keeps the norm of the bracket uniformly bounded, then two useful consequences follow: the bracket flow is defined for $t\in[0,\infty)$ and there exist convergence subsequences.  For instance, we have that
$$
r(t)=4\frac{\tr{\Ricci_{\mu^r}M_{\mu_{\pg}^r}}}{\|\mu_{\pg}^r\|^2}
$$
gives $\|\mu_{\pg}^r\|\equiv \|\mu_{\pg}(0)\|$ (see Theorem \ref{eqsrn}, (vii)).
\end{example}

\begin{example}\label{norm-Ric}
In order to prevent convergence to a flat metric without fixing a sign for the scalar curvature as in Example \ref{norm-R}, which also leaves out the metrics with $R=0$, we can use that a homogeneous manifold is flat if and only if it is Ricci flat (see \cite{AlkKml}), and consider the normalized bracket flow such that  $\tr{\Ricci_{\mu^r}^2}\equiv \tr{\Ricci_{\mu_0}^2}$.  Thus the solution stays uniformly far from being flat for all $t$, as soon as $g_0$ is nonflat, and any eventual limit will therefore be automatically nonflat.  It is difficult in this case to compute $r(t)$ explicitly, although we know that the normalized solution will have the form
$$
\mu^r(t)=\left(\frac{\tr{\Ricci_{\mu_0}^2}}{\tr{\Ricci_{\mu(\tau(t))}^2}}\right)^{1/4}\cdot\mu(\tau(t)),
$$
for some appropriate time reparametrization $\tau(t)$, where $\mu(s)$ is the bracket flow solution with $\mu(0)=\mu_0$.
\end{example}

\subsection{Example in dimension $3$}\label{exdim3}
We consider more in detail in this section the family of metrics given in Example \ref{ex1-3BF} with $c=0$.  Thus their brackets $\mu=\mu_{a,b,0}\in\hca_{1,3}$ are defined by
$$
\mu(X_3,Z_1)=X_2, \qquad \mu(Z_1,X_2)=X_3, \qquad \mu(X_2,X_3)=aX_1+bZ_1,
$$
each homogeneous space $(G_\mu/K_\mu,g_\mu)$ has Ricci operator and scalar curvature (see Example \ref{ex1-3ricci}) given by
$$
\Ricci_\mu=\left[\begin{smallmatrix} \unm a^2&&\\ &-\unm a^2+b& \\ &&-\unm a^2+b\end{smallmatrix}\right], \quad
R(\mu)=-\unm a^2+2b,
$$
and the bracket flow is equivalent to the ODE system
\begin{equation}\label{odedim3}
\left\{\begin{array}{l}
\ddt a=(-\tfrac{3}{2}a^2+2b)a, \\ \\
\ddt b=(-a^2+2b)b.
\end{array}\right.
\end{equation}

\begin{figure}[htbp]
\centering
\includegraphics[scale=1]{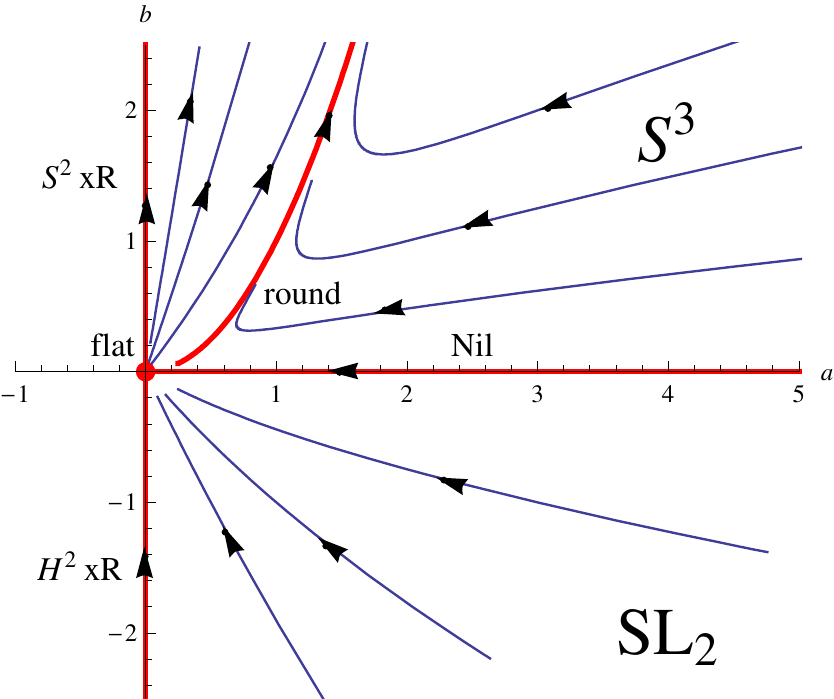}
\caption{Phase plane for the ODE system (\ref{odedim3})}\label{phpdim3}
\end{figure}

The phase plane for this system is displayed in Figure \ref{phpdim3}, as computed in Mathematica.  Up to isometry, it is enough to assume $0\leq a$, which we do from now on, and the metrics we get can be divided in bracket flow invariant subsets as follows:
\begin{itemize}
\item $a=0$, $b>0$: product metrics on $S^2\times\RR$.

\item $a=0$, $b=0$: euclidean metric on $\RR^3$.

\item $a=0$, $b<0$: product metrics on $H^2\times\RR$.

\item $a>0$, $b>0$: left-invariant metrics on $S^3$ (Berger spheres).

\item $a>0$, $b=0$: left-invariant metrics on the Heisenberg group ($Nil$).

\item $a>0$, $b<0$: left-invariant metrics on $\widetilde{\Sl_2}(\RR)$.
\end{itemize}

It is easy to see that for $b>0$ the bracket flow solution $\mu(t)$ goes to infinity in finite time, and that the parabola $b=a^2$ is invariant, which correspond to the round metrics on $S^3$ as their Ricci operator equals $\unm a^2I$.  When $b\leq 0$, one also easily obtain that $\mu(t)|_{\pg\times\pg}\to 0$ (i.e. $(a,b)\to(0,0)$), as $t\to\infty$, which implies that $g_{\mu(t)}$ locally converges and subconverges in the pointed sense to a flat metric by Theorem \ref{convmu2} and Example \ref{lieinj}.

The region $b\geq a^2$ is invariant by the bracket flow, and for the backward flow we have that
$$
\ddt (a^2+b^2)=2(\frac{3}{2}a^2-2b)a^2+2(a^2-2b)b^2\leq -a^4-2a^2b^2\leq 0.
$$
Thus all the solutions inside this region are defined for $t\in (-\infty,0]$, and so they produce ancient solutions to the Ricci flow.  These solutions may be considered the $3$-dimensional analogous of the Angenent ovals for the curve shortening flow and the Rosenau solutions for the Ricci flow on surfaces (see \cite[Chapter 2,Section 3.3]{ChwKnp}).

Let us now consider different normalizations $\mu^r=\mu_{a,b,0}$ starting at $\mu_0=\mu_{a_0,b_0,0}$.

\no
(i) {\it Volume element} (see Example \ref{norm-vol}).  An easy computation gives $r=\tfrac{1}{6}a^2-\tfrac{2}{3}b$, and thus the volume-normalized bracket flow equation is given by
$$
\left\{\begin{array}{l}
\ddt a=\tfrac{4}{3}(b-a^2)a, \\ \\
\ddt b=\tfrac{2}{3}(b-a^2)b.
\end{array}\right.
$$
It is easy to check that $\ddt \frac{a}{b^2}\equiv 0$, which gives $a=\alpha b^2$ for any $b\ne 0$ if we assume $a_0=\alpha b_0^2$, for some $\alpha>0$.  This implies that
$$
\ddt b = \tfrac{2}{3}b^2(1-\alpha^2b^3),
$$
from it can be deduced that the only fixed points are $(a,b)=(0,0)$ (flat) and $(a,b)=(\alpha^{-\frac{1}{3}},\alpha^{-\frac{2}{3}})$ (round metric on $S^3$), and that $b$ decreases for $1<b$ and increases otherwise.  An alternative way to get this is by using that $R$ must always increase forward in time for any volume-normalized solution, and we have that $R(b)=-\unm\alpha^2b^4+2b$ (see Figure \ref{voldim3} for the case $\alpha=1$).  By Theorem \ref{convmu2}, (iv) and Example \ref{lieinj}, we have that the volume-normalized Ricci flow solutions on $S^3$ (i.e. $b>0$) all converge in the pointed sense to a round metric (i.e. $b=a^2$).  On the other hand, on $\widetilde{\Sl_2}(\RR)$ (i.e. $b<0$), the volume-normalized solutions still locally converge (or subconverge in the pointed sense) to a flat metric.

\begin{figure}[htbp]
\centering
\includegraphics[scale=0.65]{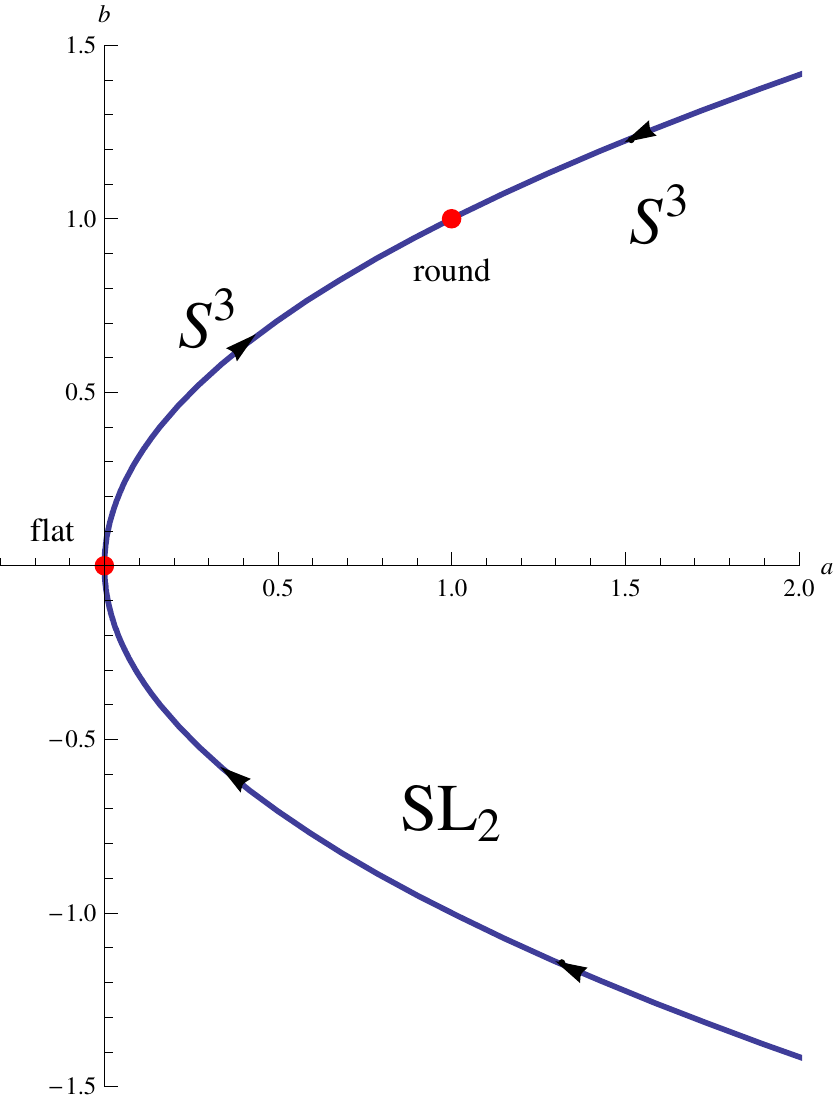}
\hspace{1cm}
\includegraphics[scale=0.8]{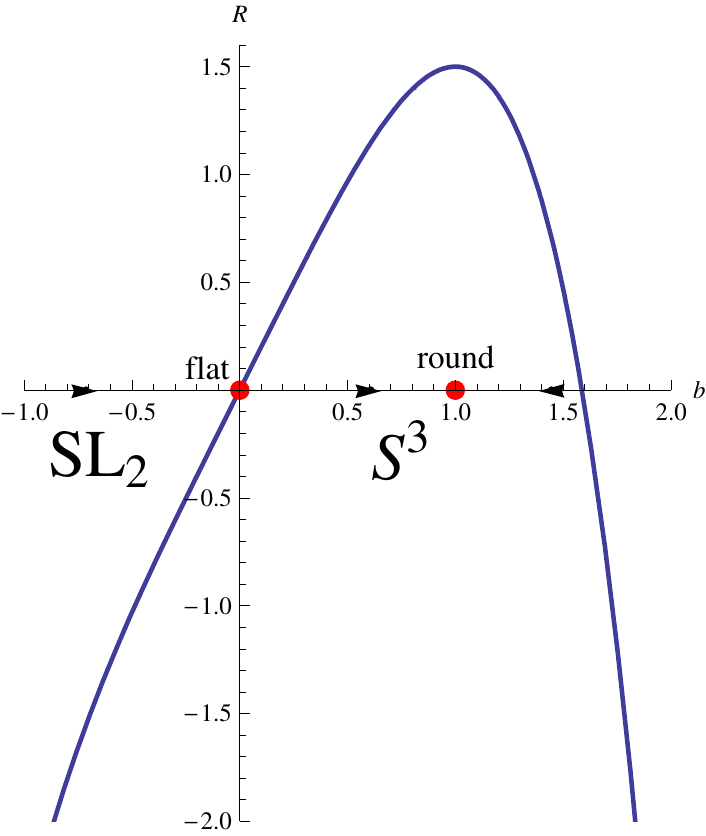}
\caption{Volume-normalized bracket flow and scalar curvature behavior}\label{voldim3}
\end{figure}

\no
(iii) {\it Positive scalar curvature} (see Example \ref{norm-R}).  By taking $r=-\tfrac{1}{2}a^4-\tfrac{4}{3}b^2+\tfrac{4}{3}a^2b$ we obtain
$R(\mu^r)=-\unm a^2+2b\equiv \frac{3}{2}$, from which follows that
$$
\ddt a=\unm(-\tfrac{1}{4}a^4-\unm a^2+\tfrac{3}{4})a.
$$
Thus $(a,b)=(0,\frac{3}{4})$ ($S^2\times\RR$) and $(a,b)=(1,1)$ (round metric on $S^3$) are the only fixed points and $a$ is increasing for $0<a<1$ and decreasing for $1<a<\infty$.  Forward in time, this gives the same convergence behavior as for the volume-normalized flow above (see Figure \ref{Rdim3}), but it produces $S^2\times\RR$ as a limit a $t\to\-infty$.  It follows that the ancient solutions (i.e. $b\geq a^2$) `connect' $S^2\times\RR$ with $S^3$ in the sense considered in \cite{BksKngNi}.

\begin{figure}[htbp]
\centering
\includegraphics[scale=0.83]{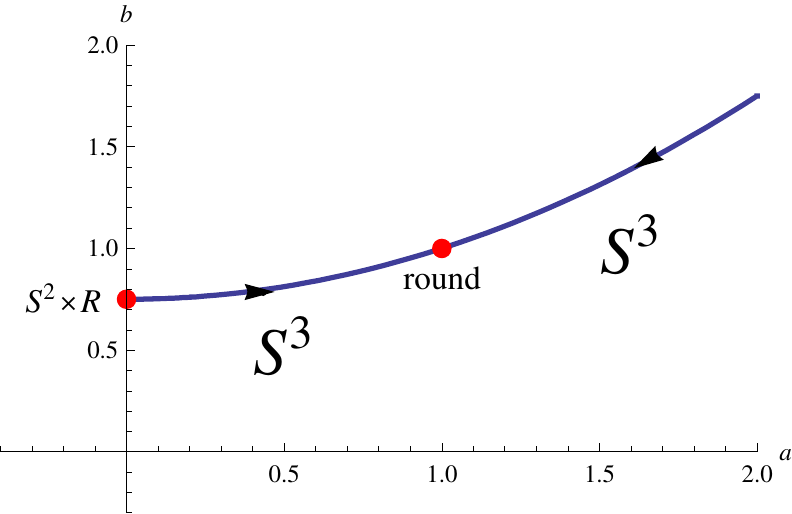}
\hspace{1cm}
\includegraphics[scale=0.85]{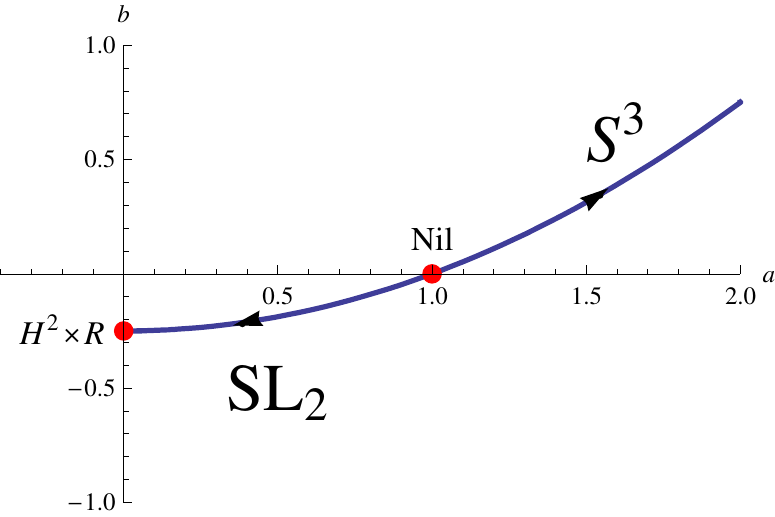}
\caption{$R$-normalized bracket flows: $R\equiv \frac{3}{2}$ and $R\equiv-\unm$.}\label{Rdim3}
\end{figure}

\no
(iv) {\it Negative scalar curvature} (see Example \ref{norm-R}).  In order to get $R(\mu^r)=-\unm a^2+2b\equiv -\frac{1}{2}$ we must consider
$r=(\tfrac{3}{2}a^4+4b^2-4a^2b)$.  It follows that
$$
\ddt a=\unm(\tfrac{1}{2}a^4-\tfrac{1}{4}a^2-\tfrac{1}{4})a,
$$
and so the fixed points are $(a,b)=(0,-\frac{1}{4})$ ($H^2\times\RR$) and $(a,b)=(1,0)$ ($Nil$).  Moreover, we obtain that $a$ is decreasing for $0<a<1$ and increasing for $1<a<\infty$.   This implies that on $S^3$ (i.e. $1<a$), the solution $\mu^r(t)$ goes to $\infty$, as $t\to\infty$.  However, the $R$-normalized solution on $\widetilde{\Sl_2}(\RR)$ (i.e. $a<1$) satisfies
$$
\lim_{t\to\infty} \mu^r(t) = \mu_{0,-\frac{1}{4},0},
$$
and so we get $H^2\times\RR$ in the limit (see Figure \ref{Rdim3}).  It follows from Theorems \ref{eqflrn} and \ref{convmu2} (and Example \ref{lieinj}) that the $R$-normalized ($R\equiv -\frac{1}{2}$) Ricci flow solution $g^r(t)$ of left-invariant metrics on $\widetilde{\Sl_2}(\RR)$ locally converges to $H^2\times\RR$ up to pullback by time-dependent diffeomorphisms, and also that there exists a subsequence $g^r(t_k)$ converging in the pointed sense to $H^2\times\RR$.  This convergence behavior was proved to hold in \cite[3.3.5]{Ltt} and \cite[4.3.1]{Glc} for certain rescalings of the Ricci flow solution starting at any left-invariant metric on $\widetilde{\Sl_2}(\RR)$.  Recall that $H^2\times\RR$ is not a unimodular Lie group, and thus this convergence can never appear in the study of the bracket flow on $\lca_3=\hca_{0,3}$ (since unimodularity is a closed condition on $\lca_3$).  However, we have just showed that it does on $\hca_{1,3}$ (compare with \cite[Section 6]{GlcPyn}).

\section{Example on compact Lie groups}\label{exss}

In this section, we study the Ricci flow on certain $2$-parameter family of left-invariant metrics on semisimple Lie groups.

Let $\ggo$ be a compact simple Lie algebra of dimension $n$, with Lie bracket denoted by $\lb$.  Endow $\ggo$ with the inner product $\ip$ given by minus the Killing form of $\ggo$, that is, $B(X,X)=-1$ for any $X\in\ggo$, $\| X\|=1$.  This implies that $(\ad{X})^t=-\ad{X}$ for any $X$, from which it follows that the moment map equals $M=-\unc I$, and so the Ricci operator of the bi-invariant metric defined by $\ip$ on the corresponding simply connected Lie group $G$ is given by $\Ricci=\unc I$.  Consider a Cartan decomposition
$\ggo=\hg\oplus\mg$, $\dim{\hg}=h$, $\dim{\mg}=m$ (i.e. $[\hg,\hg]\subset\hg$, $[\hg,\mg]\subset\mg$ and $[\mg,\mg]\subset\hg$), which is automatically orthogonal with respect to $\ip$, and let us make the following assumption: the Killing form $B_{\hg}$ is a multiple of $B$ restricted to $\hg$, say
\begin{equation}\label{killh}
B_{\hg}=\alpha B|_{\hg\times\hg}, \qquad 0\leq\alpha<1.
\end{equation}
It is easy to see that $\alpha=\frac{2h-m}{2h}$ (see \cite[Theorem 11,(i)]{DtrZll}).  This situation holds for instance for any irreducible symmetric space $G/H$ with $H$ simple or abelian, and so $G$ can be any compact simple Lie group in this section, with the only exception of $\Spe(2k+1)$ (see e.g. the list \cite[Table 7.102]{Bss}).

Let $\mu=\mu_{a,b}$ denote the Lie bracket on the vector space $\ggo$ defined by
$$
\mu|_{\hg\times\hg}=a\lb, \qquad \mu|_{\hg\times\mg}=a\lb, \qquad \mu|_{\mg\times\mg}=b\lb.
$$
We note that $(\ggo,\mu)$ is isomorphic to $\ggo$ for $a,b>0$, and to a noncompact simple Lie algebra $\ggo_{nc}$ for $b<0<a$, which is another real form of the complexification $\ggo_{\CC}$.  As $\mu_{a,b}$ is isometric to $\mu_{-a,-b}$, it is enough to assume $0\leq a$, and by using the Killing form below, it is easy to see that the left-invariant metrics we get can be arranged according to their underlying Lie group as follows (any item defines a bracket flow invariant subset):

\begin{itemize}
\item[(i)] $a=0$, $b\ne 0$: a $2$-step nilpotent Lie group $N$ with $h$-dimensional derived algebra equal to its center.

\item[(ii)] $a=0$, $b=0$: abelian group $\RR^n$ (euclidean metric).

\item[(iii)] $a>0$, $b>0$: compact simple Lie group $G$.

\item[(iv)] $a>0$, $b=0$: a semidirect product $H\ltimes\RR^m$ (these metrics are actually isometric to the product $H\times\RR^m$ of a bi-invariant metric on $H$ and the euclidean metric on $\RR^m$).

\item [(v)] $a>0$, $b<0$: noncompact simple Lie group $G_{nc}$ with Lie algebra $\ggo_{nc}$.
\end{itemize}

Let $(G_\mu,\ip)$ be the Lie group endowed with a left-invariant metric according to (\ref{hsmu}), corresponding to $\mu=\mu_{a,b}\in\lca_n=\hca_{0,n}$.  It follows from (\ref{vol}) and Proposition \ref{const}, (ii) that for $a,b>0$ (resp. $a>0$, $b<0$), $(G_\mu,\ip)$ is isometric to the left invariant metric on $G$ (resp. $G_{nc}$) defined by the inner product
$$
\ip_{a,b}:=\frac{1}{a^2}\ip|_{\hg\times\hg} + \frac{1}{a|b|}\ip|_{\mg\times\mg}.
$$
It is known that $(G_\mu,\ip)$ is a naturally reductive homogeneous manifold for any $a,b$, with respect to its presentation as a homogeneous space $G\times H/H$ (see \cite[Theorem 1]{DtrZll} for the semisimple cases (iii), (v), and \cite{manus} for the nilpotent case (i)).  Furthermore, it easily follows from \cite[Theorem 3]{DtrZll} that $(G_\mu,\ip)=G\times H/H$ is normal homogeneous if and only if $0\leq b\leq a$, and from \cite[Theorem 9]{DtrZll} that $(G_\mu,\ip)$ has nonnegative sectional curvature if and only if $0\leq b\leq \frac{4}{3}a$.

It is straightforward to see by using (\ref{killh}) that the Killing form and moment map of $(G_\mu,\ip)$ are respectively given by
$$
B_\mu=\left[\begin{smallmatrix} -a^2I&\\ &-abI \end{smallmatrix}\right], \qquad
M_\mu=\left[\begin{smallmatrix} -\frac{a^2}{2}I+\frac{a^2}{4}\alpha I+\frac{b^2}{4}(1-\alpha)I&\\ &(\unm ab-\unc b^2)I \end{smallmatrix}\right].
$$
The Ricci operator and scalar curvature are therefore given by
$$
\Ricci_\mu=\unc\left[\begin{smallmatrix} (\alpha a^2+(1-\alpha)b^2)I&\\ &(2ab-b^2)I \end{smallmatrix}\right],  \qquad R=\frac{2h-m}{8}a^2-\frac{m}{8}b^2+\frac{m}{2}ab,
$$
and an easy computation gives that
$$
\begin{array}{l}
\pi(\Ricci_\mu)\mu|_{\hg\times\ggo}=-\unc (\alpha a^2+(1-\alpha)b^2)a\, \lb, \\ \\
\pi(\Ricci_\mu)\mu|_{\mg\times\mg}=\unc (\alpha a^2+(3-\alpha)b^2-4ab)b\, \lb.
\end{array}
$$
This implies that the bracket flow is equivalent to the following ODE system for $a=a(t)$, $b=b(t)$:
\begin{equation}\label{bfexss}
\left\{\begin{array}{l}
\ddt a = \unc (\alpha a^2+(1-\alpha)b^2)a, \\ \\
\ddt b = -\unc (\alpha a^2+(3-\alpha)b^2-4ab)b.
\end{array}\right.
\end{equation}
The phase plane for this system is displayed in Figure \ref{phpss}, as computed in Mathematica, in the particular case $\ggo=\sug(3)$, $\hg=\sug(2)$ (where $h=3$, $m=5$ and $\alpha=\frac{1}{6}$).

\begin{figure}[htbp]
\centering
\includegraphics[scale=1]{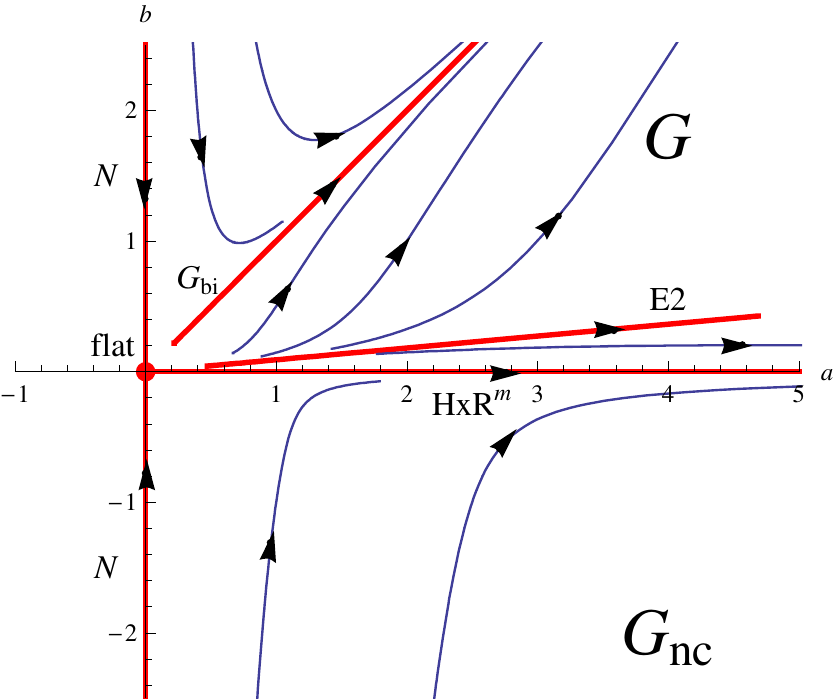}
\caption{Phase plane for the ODE system (\ref{bfexss})}\label{phpss}
\end{figure}

We have that $(G_\mu,\ip)$ is Einstein if and only if $\alpha a^2+(1-\alpha)b^2=2ab-b^2$, whose solutions are $b=a$ (i.e. bi-invariant metrics on $G$) and $b=\frac{\alpha}{2-\alpha}a$.  The other Ricci solitons are $H\times\RR^m$, as it is a product of an Einstein manifold and an euclidean space, and the nilsoliton $N$ (see \cite{soliton}), whose Ricci operator satisfies
$$
\Ricci_{\mu_{0,b}}= -\unc((1-\alpha)b^2+2b^2)I+\left[\begin{smallmatrix} \unm((1-\alpha)b^2+b^2)I & \\ &\unc((1-\alpha)b^2+b^2)I\end{smallmatrix}\right] \in\RR I\oplus\Der(\ggo,\mu_{0,b}).
$$
It follows that the region $0\leq b\leq a$ is invariant by the bracket flow, and since $\ddt a\geq 0$ on it the backward flow stays forever bounded.
Thus all the solutions inside this region are defined for $(-\infty,0]$, and so this gives explicit examples of left-invariant ancient solutions to the Ricci flow on any compact simple Lie group different from $\Spe(2k+1)$.  We actually get two non-equivalent (up to homothety) ancient solutions which are not Einstein if $\alpha\ne 0$: $\frac{\alpha}{2-\alpha}a<b<a$ and $0<b<\frac{\alpha}{2-\alpha}a$ (by using that the ratio of the two Ricci eigenvalues is respectively $>1$ and $<1$).  From the behavior of the $R$-normalized bracket flow obtained below (see Figure \ref{Rss}, left), we have that these ancient solutions flow the second Einstein metric $E2$ (i.e. $b=\frac{\alpha}{2-\alpha}a$) into the bi-invariant metric $G_{bi}$ (i.e. $b=a$) and into the Ricci soliton $H\times\RR^m$, respectively.  We have recently become aware that the same behavior has been discovered in \cite{BksKngNi} for certain ancient solutions on many homogeneous spaces including spheres and complex projective spaces, and that the construction given above is actually a particular case of the Riemannian submersion method given in \cite[Section 7]{BksKngNi}.

On the contrary, for $b>a>0$, it is easy to see that $b\to\infty$, as $t\to T$, for some finite negative time $-\infty<T<0$, and thus these are not ancient solutions.

We now consider some normalizations.

\no
(i) {\it Volume element} (see Example \ref{norm-vol}). By using that for all $a>0$, $b\ne 0$,
\begin{equation}\label{vol}
\mu_{a,b}=\left[\begin{smallmatrix} \frac{1}{a}I & \\ & \frac{1}{\sqrt{a|b|}}I\end{smallmatrix}\right]\cdot\mu_{1,\pm 1}, \qquad \pm 1:=\frac{b}{|b|},
\end{equation}
we obtain that $\vol(\mu_{a,b})=a^{-(h+m/2)}|b|^{-m/2}\vol(\mu_{1,\pm 1})$, and a region of constant volume element is therefore given by
$$
b=\pm a^{-\omega}, \qquad \omega:=\frac{m+2h}{m}>1.
$$
The scalar curvature can be written as
$$
R(a)=\unc\left(h\alpha a^2+h(1-\alpha)a^{-2\omega}\pm 2ma^{1-\omega}-ma^{-2\omega}\right),
$$
and thus
$$
R'(a)=\unm a^{-2\omega-1}p_{\pm}(a^{\omega+1}), \qquad\mbox{where}\quad p_{\pm}(x):=h\alpha x^2\pm(1-\omega)m x+\omega(m-h(1-\alpha)).
$$
It follows from \cite[Theorem 11,(i)]{DtrZll} that $\alpha\geq\frac{2h-m}{2h}>\frac{h-m}{h}$ (this also follows by using that $\|\mu_{a,b}\|^2=-4\tr{M_{\mu_{a,b}}^2}=h(2-\alpha)a^2+(m-(1-\alpha)h)b^2$), and so $p_{\pm}(0)>0$.  On the other hand, we know that $R'(a)$ and henceforth $p_{\pm}$ vanishes precisely on the $a$-coordinate of Einstein metrics (see Example \ref{norm-vol}).  This implies that in the compact case $b>0$, the zeroes of $p_+$ are $a=1,(\frac{2-\alpha}{\alpha})^{-\omega-1}$, and hence $p_+<0$ for $1<a<(\frac{2-\alpha}{\alpha})^{-\omega-1}$ and $p_+>0$ otherwise.  For the noncompact case $b<0$, one has that $p_-$ is always positive.  As $R$ must always increase, we obtain from the behavior of $p_\pm$ just described that $a$ decreases when  $1<a<(\frac{2-\alpha}{\alpha})^{-\omega-1}$ and increases otherwise if $b>0$, and that $a$ always increases if $b<0$ (see Figures \ref{volss} and \ref{volssR}).  By using Theorem \ref{convmu3}, (iv), we obtain pointed convergence to the bi-invariant metric $G_{bi}$ on $G$.  On the contrary, the other Einstein metric is unstable, and the solution on $G_{nc}$ diverges.

\begin{figure}[htbp]
\centering
\includegraphics[scale=1]{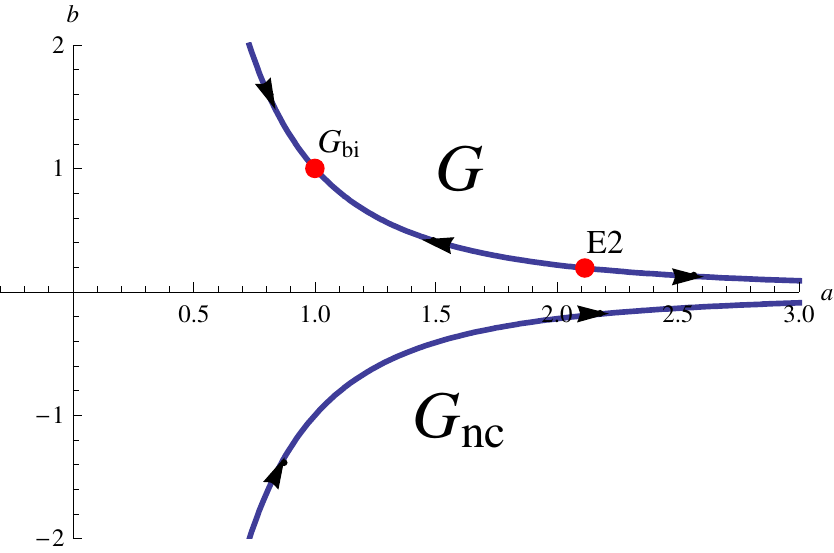}
\caption{Volume-normalized bracket flow}\label{volss}
\end{figure}

\begin{figure}[htbp]
\centering
\includegraphics[scale=0.75]{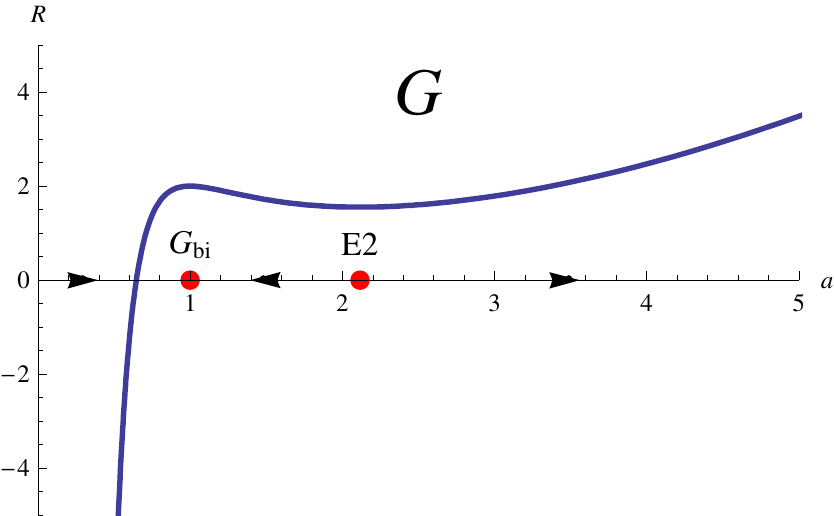}
\hspace{1cm}
\includegraphics[scale=0.75]{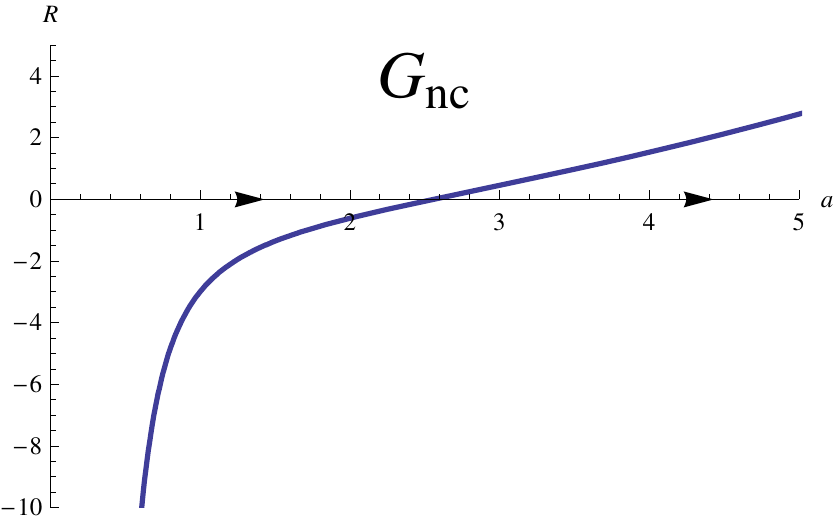}
\caption{Scalar curvature behavior}\label{volssR}
\end{figure}

\begin{figure}[htbp]
\centering
\includegraphics[scale=0.8]{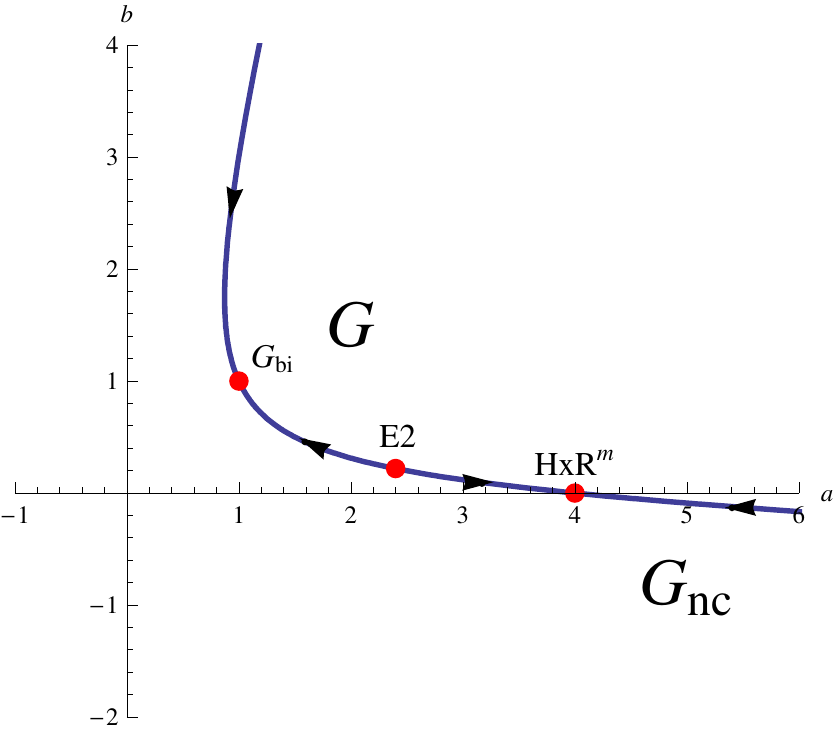}
\hspace{1cm}
\includegraphics[scale=0.76]{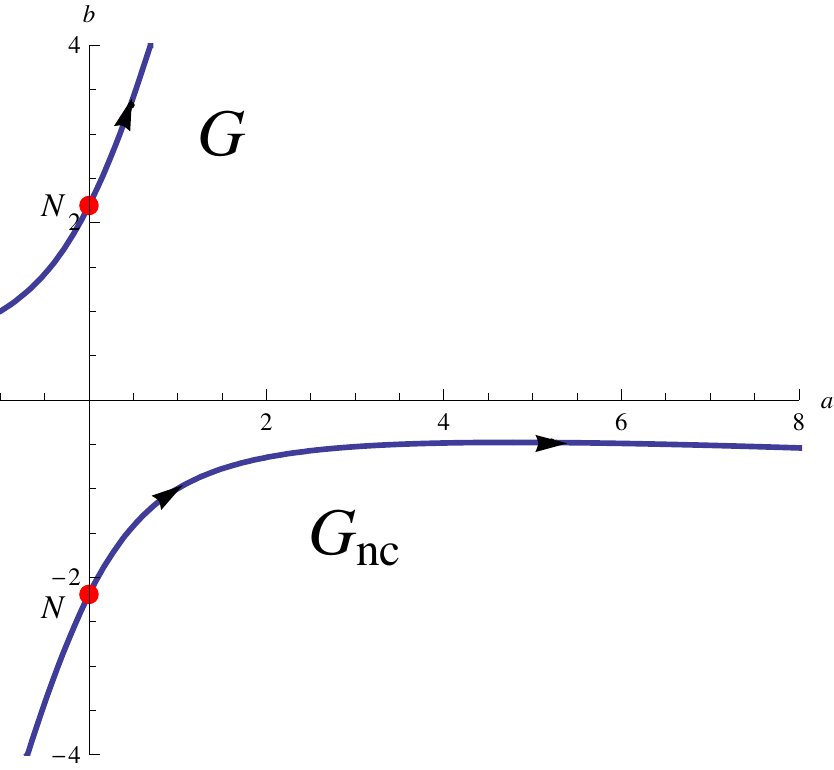}
\caption{$R$-normalized bracket flow: $R\equiv 2$ and $R\equiv -3$}\label{Rss}
\end{figure}

\no
(i) {\it Scalar curvature} (see Example \ref{norm-R}).  In order to get more convergence information, we consider the normalizations by scalar curvature:
$$
R=\unc\left((\alpha a^2+(1-\alpha)b^2)h+(2ab-b^2)m\right)\equiv R_0,
$$
which is obtained with the $r$-normalized bracket flow with
$$
r=-\frac{\tr{\Ricci^2}}{R} = -\frac{1}{16 R_0}\left((\alpha a^2+(1-\alpha)b^2)^2h+(2ab-b^2)^2m\right)
$$
In the positive case, say $R_0=2$, we replace in the equation
$$
\ddt b = -\unc (\alpha a^2+(3-\alpha)b^2-4ab)b +rb,
$$
the value of $a$ and $r$ by their expressions in terms of $b$, and deduce the behavior of the solution from the sign of $\ddt b$ (see Figure \ref{Rss}, left).  This also gives pointed convergence to $G_{bi}$, but in addition local convergence (or pointed subconvergence) to $H\times\RR^m$ of two solutions, one consisting of left-invariant metrics on $G$ and the other on $G_{nc}$.

For $R_0=-\frac{3}{2}$, we argue in the same way by writing
$$
\ddt a = \unc (\alpha a^2+(1-\alpha)b^2)a +ra,
$$
in terms of $a$ and using that $\mu_{a,b}$ is isometric to $\mu_{-a,-b}$ (see Figure \ref{Rss}, right).  This does not provide any convergence in forward time, it only shows that $N$ is unstable.

\end{document}